\documentclass[12pt]{amsart}
\usepackage{amssymb,amsmath}
\usepackage{geometry}    
\usepackage{mathrsfs}

\usepackage{graphicx}

\usepackage{vmargin}
\setmargrb{1in}{1in}{1in}{1in}

\newtheorem{theorem}{Theorem}[section]
\newtheorem{lemma}[theorem]{Lemma}
\newtheorem{proposition}[theorem]{Proposition}
\newtheorem{corollary}{Corollary}
\theoremstyle{definition}
\newtheorem{definition}{Definition}
\newtheorem{remark}{Remark}

\newtheorem{problem}{Problem}

\newcommand{\scp}[1]{\langle#1\rangle}

\newcommand{\ux}{\underline{\xi}}

\newcommand{\om}{\omega}
\newcommand{\wo}{\widehat{\omega}}

\newcommand{\Rm}{\mathbb{R}^m}
\newcommand{\Rr}{\mathscr{R}_m}
\newcommand{\R}{\mathbb{R}}
\newcommand{\ut}{\underline{\theta}}
\newcommand{\Rd}{\mathscr{D}_m}

\begin{document}

\title[Metric results on inhomogeneously singular vectors]{Metric results on inhomogeneously singular vectors}

\author{Johannes Schleischitz}


\thanks{Middle East Technical University, Northern Cyprus Campus, Kalkanli, G\"uzelyurt \\
    johannes@metu.edu.tr ; jschleischitz@outlook.com}



\begin{abstract}
     Given $\ut\in\Rm$ and any norm $\Vert.\Vert$ on $\Rm$, we consider ``inhomogeneously
    singular'' vectors in $\Rm$ 
    that admit an integer vector
    solution $(q,\underline{p})=(q,p_1,\ldots,p_m)$
    to the system
    \[
    1\leq q\leq Q, \qquad \Vert q\ux-\underline{p}-\underline{\theta}\Vert\leq cQ^{-1/m}
    \]
    for any $c>0$ and all large $Q$. 
    We show that this set
    has large packing dimension,
    and in view of recent deep results by Das, Fishman, Simmons, Urba\'nski,
    our lower bounds are almost sharp (up to $O(m^{-1})$). 
    We establish slightly weaker bounds for the 
    Hausdorff dimension as well.
    Our bounds
    are applicable to the $b$-ary setting, i.e. when restricting to
    $q$ above integral powers of some given base $b\geq 2$.
    We further derive similar results for vectors on
    certain $m$-dimensional fractals, thereby contributing to a question
    of Bugeaud, Cheung and Chevallier and complementing recent
    work by Kleinbock, Moshchevitin and Weiss and by Khalil. 
    Moreover, we show that in contrast to Liouville vectors,
    the set of singular vectors in $\Rm$ 
    does not form a comeagre set. We infer this
    from a general new result that any comeagre set
    in $\Rm$ has full packing dimension. As another independent
    consequence of our method, we show that there are subsets $A,B$ of
    $\mathbb{R}^m$ for which $A+B=B+B=\mathbb{R}^m$ 
    but $A+B$ has Hausdorff
    dimension less than $m$, and generalizations. 
    The proofs 
    rely on observations on sumsets and a result by Tricot
    involving Cartesian products and are surprisingly elementary. 
    The topological results further use an observation of Erd\H{o}s.   
\end{abstract}

\maketitle

{\footnotesize{

Math Subject Classification 2020: 11J13, 11J20, 11J82, 11J83}}



\section{Introduction}  \label{int01}

\subsection{Homogeneous approximation and singular vectors} \label{s1.1}

Throughout, let $m$ be a positive integer, mostly $m\geq 2$,
and let $\Vert.\Vert$ be any norm on $\Rm$.
By Dirichlet's Theorem, for any $\ux\in\Rm$ and any $Q\geq 1$ the homogeneous 
system of inequalities
	\begin{equation} \label{eq:1}
1\leq q\leq Q, \qquad \Vert q\ux-\underline{p}\Vert \leq cQ^{-1/m}
\end{equation}
has a solution in an integer vector $(q,\underline{p})=(q,p_1,\ldots,p_m)$
for $c=1$.
As customary we call a vector $\ux\in\Rm$ singular 
if \eqref{eq:1} has a solution for 
arbitrarily small $c>0$ and all large $Q\geq Q_0(c)$. 
We will frequently restrict to totally irrational vectors,
i.e. to the set $\Rr\subseteq \Rm$ of $\ux=(\xi_1,\ldots,\xi_m)$ for which $\{1,\xi_1,\ldots, \xi_m\}$ is $\mathbb{Q}$-linearly independent,
or equivalently that do not lie in a rational affine hyperplane.
Denote by $\mathcal{S}_m\subseteq \Rm$ resp. $\mathcal{S}_m^{\ast}=\mathcal{S}_m\cap \Rr$
the set of singular resp. 
totally irrational singular vectors in $\Rm$.
Clearly $\mathcal{S}_m^{\ast}\subseteq \mathcal{S}_m$, and for $m=1$ we have $\mathcal{S}_1=\mathbb{Q}$ 
and $\mathcal{S}_1^{\ast}=\emptyset$ due to Khintchine~\cite{khintchine}.

An important problem in Diophantine approximation that has received 
much attention
in recent years is to determine the Hausdorff and packing dimensions
(for a definition see see~\cite{falconer}) 
of the set of singular vectors (matrices).
Building up on a landmark paper by 
Cheung~\cite{che} for $m=2$, Cheung and Chevallier~\cite{cheche}
determined the Hausdorff dimension of $\mathcal{S}_m$ for $m\geq 2$.
An upper bound in a more general matrix setting
was later obtained by 
Kadyrov, Kleinbock, Lindenstrauss and Margulis~\cite{kklm}.
Using their deep variational principle, 
Das, Fischman, Simmons, Urba\'nski~\cite{dfsu1, dfsu2} showed that this upper bound is sharp
and thereby extended~\cite{cheche}, with a new proof.
In~\cite{dfsu1, dfsu2} it is
also shown that the packing dimensions
attain the same value (in the general matrix setting). Denoting by $\dim_H(A)$ resp. $\dim_P(A)$ the Hausdorff and packing dimension
of $A\subseteq \Rm$, the claims read as follows.

\begin{theorem}[Cheung, Chevallier; Das, Fishman, Simmons, Urba\'nski] \label{DFSU}
	For any $m\geq 2$, we have
	\[
	\dim_{H}(\mathcal{S}_m)=\dim_{H}(\mathcal{S}_m^{\ast})=\dim_{P}(\mathcal{S}_m)=\dim_{P}(\mathcal{S}_m^{\ast})=m-1+\frac{1}{m+1}.
	\]
\end{theorem}

For $w\in[0,1)$, define $\mathcal{S}_m(w)$ resp. $\mathcal{S}_m^{\ast}(w)$ 
as the set of $\ux$ for which 
\begin{equation}  \label{eq:c=1}
1\leq q\leq Q, \qquad \Vert q\ux-\underline{p}\Vert \leq cQ^{-w}
\end{equation}
is soluble in integers for any $c>0$ and $Q\geq Q_0(c)$, 
for $\ux\in \Rm$ resp. $\ux\in\Rr$.
Then
\begin{equation}  \label{eq:GONG}
\mathcal{S}_m=\mathcal{S}_m(1/m), \qquad \mathcal{S}_m^{\ast}=\mathcal{S}_m^{\ast}(1/m), 
\end{equation}
and we need not consider $w\geq 1$ since then $\mathcal{S}_m(w)=\mathbb{Q}^m$, $\mathcal{S}_m^{\ast}(w)=\emptyset$
by~\cite{khintchine} mentioned above. 
Despite Theorem~\ref{DFSU}, the following problem
discussed in a slightly altered form (and in the general matrix setting) in~\cite{dfsu1}, 
remains open for $m>2$.

\begin{problem} \label{p1}
	For $m\geq 2$ and $w\in[1/m,1)$, determine the Hausdorff and packing dimensions of the sets
$\mathcal{S}_{m}(w)$ and $\mathcal{S}_m^{\ast}(w)$.
\end{problem}

In view of \eqref{eq:GONG}, for $w=1/m$ Problem~\ref{p1} 
is solved by Theorem~\ref{DFSU}.
For $m=2$, after several preceeding papers,
it is completely solved as well 
in view of Theorem~1.9 and the subsequent Remark in~\cite{dfsu1}, however the explicit formulas are cumbersome.
Lower bounds for general $m$ and the packing dimensions
given in~\cite[Theorem~3.8]{dfsu2} can be phrased
\begin{equation}  \label{eq:dfsuoben}
\dim_P(\mathcal{S}_m(w))\geq \max\left\{ m-1+\frac{1}{m+1}- \frac{(2m+1)(mw-1)}{(m+1)(w+1)}, m-\frac{m(m-1)w}{m-w} \right\}.
\end{equation}
For $m=2$ and regarding Hausdorff dimension, see also the independent
work by Bugeaud, Cheung and Chevallier~\cite{bcc} who showed that
\begin{equation}  \label{eq:bchch}
\dim_H(\mathcal{S}_2(w))\geq 2-2w, \qquad w\in(1/2,1), 
\end{equation}
and also settled equality when $w>1/\sqrt{2}$.
For $m>2$, the precise formulas for both Hausdorff and packing 
dimension of $\mathcal{S}_m(w)$ remain unknown, however
several other estimates complementing the stated claims
are given in~\cite{dfsu1, dfsu2}. 

Next we introduce
the classical homogeneous exponents of ordinary approximation $\om_m(\ux)$ respectively of uniform approximation
$\wo_m(\ux)$ as the supremum of $w$ such that
the system \eqref{eq:c=1} for $c=1$
admits a solution in integers for arbitrarily large resp. all large $Q$. 
For any $\ux\in\Rm$, by \eqref{eq:1} we obviously have
\begin{equation}  \label{eq:kontrast}
\om_m(\ux) \geq \wo_m(\ux)\geq \frac{1}{m}.
\end{equation}
For any $w$ the sets $\mathcal{S}_m(w)$
have a simple relation to superlevelsets of
$\wo_m$ given by
\begin{equation}  \label{eq:accext}
\{ \ux\in\Rm: \wo_m(\ux)>w \}\subseteq 
\mathcal{S}_m(w)\subseteq \{ \ux\in\Rm: \wo_m(\ux)\geq w \},
\end{equation}
and similarly for $\mathcal{S}_m^{\ast}(w)$.
In~\cite{dfsu1, dfsu2} the notation $Sing(m,1)$ means
the singular vectors in $\Rm$, which is our $\mathcal{S}_m$.
Thus our sets $\mathcal{S}_m^{\ast}$ coincide with $Sing(m,1)\cap \Rr$. 
The sets $\widetilde{Sing}_{m,1}^{\ast}(w):=\{ \ux\in\Rr: \wo_m(\ux)\geq w \}$ studied in~\cite{dfsu1} in view of \eqref{eq:accext}
just subtly differ 
from our sets $\mathcal{S}_m^{\ast}(w)$. 
It will be convenient for us to only impose lower bounds for the
order of uniform approximation. However,
certain results below can be extended with some effort
to sets $Sing_{m,1}^{\ast}(w):=\{ \ux\in\Rr: \wo_m(\ux)=w \}$
or $Sing_{m,1}(w):=\{ \ux\in\Rm: \wo_m(\ux)=w \}$ from~\cite{dfsu1} as well, see Remark~\ref{hirsch} below. Presumably, the metrical formulas
for levelsets and superlevelsets coindice, see~\cite[Theorem~4.9]{dfsu2}
and subsequent comments.  

The natural restriction to $\Rr$, inducing the exponents
with the asterisk ''$\ast$'', is crucial at some places. 
This applies in particular to our Theorem~\ref{cantor} on (homogeneously) singular vectors in certain fractal sets.  
Indeed, for $1\leq k\leq m$ integers, $\underline{\zeta}=(\zeta_1,\ldots,\zeta_k)\in\mathbb{R}^k$
and $\ux=(\zeta_1,\ldots,\zeta_k,\zeta_k,\ldots,\zeta_k)\in\Rm$
have the same irrationality exponents, i.e. $\om_k(\underline{\zeta})=\om_m(\ux)$ and $\wo_k(\underline{\zeta})=\wo_m(\ux)$, and similarly for 
the $b$-ary exponents. From Dirichlet's Theorem
and since the Hausdorff dimension of a set never exceeds its packing dimension~\cite{falconer} we get 
\begin{equation}  \label{eq:tr}
\dim_P(\mathcal{S}_m(w)\setminus \Rr) \geq \dim_H(\mathcal{S}_m(w)\setminus \Rr) \geq k, \qquad w<\frac{1}{k}.
\end{equation}
Similar results can also be derived from transference arguments,
considering a linear form in $m$ variables.
%
At some places \eqref{eq:tr} 
exceeds our lower bounds, thereby directly implying some of 
our (homogeneous, non-$b$-ary) results without restriction to $\Rr$.

\subsection{Inhomogeneous and $b$-ary approximation}  \label{inhomo}

Consider now the classical inhomogeneous approximation problem,
where for given $\ux, \ut\in \Rm$,
we study small values of $\Vert q\ux-\underline{p}-\ut\Vert$
in terms of $Q$, where $1\leq q\leq Q$. 
The theory again splits in ordinary and uniform approximation.
Define the ordinary exponent $\om_{m,\ut}(\ux)$
resp. uniform exponent $\wo_{m,\ut}(\ux)$ of inhomogeneous approximation
as the supremum
of $w$ such that
\begin{equation} \label{eq:F}
1\leq q\leq Q, \qquad \Vert q\ux-\underline{p}-\ut\Vert \leq
c Q^{-w}
\end{equation}
has a solution in a non-zero integer vector $(q,\underline{p})=(q,p_1,\ldots,p_m)$ 
for 
$c=1$ and 
certain arbitrarily large $Q$ resp. all large $Q$.
The homogeneous exponents from \S~\ref{s1.1} 
represent the special case
\begin{equation} \label{eq:konserve}
\om_{m}(\ux)= \om_{m,\underline{0}}(\ux), \qquad \wo_{m}(\ux)= \wo_{m,\underline{0}}(\ux).
\end{equation}
In the inhomogeneous case, hugely contrasting \eqref{eq:kontrast}, 
even the ordinary exponents may
vanish. Yet,
from a metrical point of view, 
for ordinary approximation
we still have a very satisfactory understanding 
by the 
inhomogeneous Khintchine-Groshev Theorem, analogous
to the homogeneous theory.
See~\cite{allen, hss} for recent
subtle improvements and further references. On the other hand,
little is known when extending Problem~\ref{p1} and related metrical questions on uniform approximation to the 
inhomogeneous setting.  These problems are
explictly discussed in~\cite[\S~5.8]{dfsu2}.  
However, we should mention
a special case of a metrical result by T. Kim and W. Kim~\cite[Corollary~1.5]{kimkim}.

\begin{theorem}[Kim, Kim]  \label{blau}
	For $m\geq 2$ and every $\ut\notin \mathbb{Q}^m$ and
	every $w>0$, we have
	\[
	\dim_H( \{ \ux\in\Rm: \wo_{m,\ut}(\ux) \leq w \} ) = \min\{ m-\frac{1-mw}{1+w} , m\}. 
	\]
\end{theorem}

The same is true when taking equality in the left hand side set.
See also Bakhtawar 
and Simmons~\cite{baksim} for refinements, and for 
further contributions to Problem~\ref{p1} Bugeaud and Laurent~\cite{buglau} 
and Kim and Liao~\cite{kimliao}, the last treating
$m=1$ and fixing $\ux$ and letting $\ut$ vary. Note that
we are primarily interested in sets as in Problem~\ref{p1}
with the reverse inequality and $w>1/m$, where Theorem~\ref{blau}
gives little information. 
In this paper we develop some metrical theory
for inhomogeneous, uniform approximation in arbitrary dimension $m$.


In contrast to the homogeneous case in \S~\ref{s1.1},
for general $\ut$, we are unable to provide an affine hyperplane that
consists of vectors with good (even ordinary) 
approximation of order greater than $1/m$, so it seems to us that
no inhomogeneous analogue of \eqref{eq:tr} exists.
Nevertheless, to be on the safe side, we will often restrict to
$\Rd$ which we introduce 
as the complement of any given countable union of affine hyperplanes
in $\Rm$.

Besides the inhomogeneous setting, in this paper we also consider vectors with good $b$-ary approximation, i.e.
where $q=b^N, N\in\mathbb{N}$, for some $b\geq 2$. Then the base $b$ expansion of all coordinates of $b^N\ux-\underline{p}-\ut\in\Rm$ 
start with long blocks of consecutive $0$ or $b-1$ digits. Define the inhomogeneous, 
$b$-ary exponents $\om_{m,\ut}^{(b)}(\ux)$ 
and $\wo_{m,\ut}^{(b)}(\ux)$ as above but where we take 
$q=b^N$ integer
powers of $b$ in \eqref{eq:F}. The homogeneous, ordinary $b$-ary exponent $\om_1^{(b)}(\xi)=\om_{1,0}^{(b)}(\xi)$
equals the exponent $v_b(\xi)$ 
defined and studied by Amou and Bugeaud~\cite{ab10}. It was also implicitly studied by Levesley, Salp and Velani~\cite{lsv} for $b=3$ and numbers in Cantor's middle third set. In contrast to \eqref{eq:kontrast}, the $b$-ary exponent $\om_m^{(b)}$
generically (for Lebesgue almost all
$\ux\in\Rm$) 
vanishes. 
In the 
inhomogeneous setting,
the inequality $\wo_{1,\theta}(\xi)>1$ occurs
for certain $\xi,\theta\in\mathbb{R}$ not inducing the trivial
case given by $q\xi-p-\theta=0$ for some
$(q,p)\in\mathbb{Z}^{2}$, see Kim and Liao~\cite{kimliao} for metrical results. It is however unclear to us whether
$\wo_{1,\theta}^{(b)}(\xi)>1$ may occur unless in the trivial situation. 
We extend some more notions from \S~\ref{s1.1}
to our inhomogeneous, $b$-ary setting.

\begin{definition}  \label{d}
	For $w\in[0,\infty)$, define the nested sets
	\[
	\mathcal{S}_{m,\ut}(w)= \{ \ux\in\Rm:\; \eqref{eq:F} \; \text{has a solution for every} \;\; c>0 \;\; \text{and all large Q} \},
	\] 
	and let $\mathcal{S}_{m,\ut}^{\ast}(w)=\mathcal{S}_{m,\ut}(w)\cap \Rd $ with $\Rd$ defined above. Further let
	\[
	\mathcal{S}_{m,\ut}= \mathcal{S}_{m,\ut}(1/m),\qquad
	\mathcal{S}_{m,\ut}^{\ast}= \mathcal{S}_{m,\ut}^{\ast}(1/m).
	\] 
	Derive accordingly
	\[
	\mathcal{S}_{m,\ut}^{(b)}(w)\subseteq \mathcal{S}_{m,\ut}(w),\quad 
	\mathcal{S}_{m,\ut}^{(b)}\subseteq \mathcal{S}_{m,\ut}, \quad
	\mathcal{S}_{m,\ut}^{(b)\ast}(w)\subseteq \mathcal{S}_{m,\ut}^{\ast}(w),
	\quad \mathcal{S}_{m,\ut}^{(b)\ast}\subseteq \mathcal{S}_{m,\ut}^{\ast}
	\]
	with respect to $b$-ary approximation, i.e. where we restrict
	to $q=b^N$.
\end{definition}

As a consequence of \eqref{eq:konserve}, the according extension of \eqref{eq:accext} holds, and
we find the classical singular vectors as a special
case of our definition via
\[
\mathcal{S}_m(w)= \mathcal{S}_{m,\underline{0}}(w), \qquad \mathcal{S}_m= \mathcal{S}_{m,\underline{0}}, \qquad \mathcal{S}_m^{\ast}(w)= \mathcal{S}_{m,\underline{0}}^{\ast}(w), \qquad \mathcal{S}_m^{\ast}= \mathcal{S}_{m,\underline{0}}^{\ast},
\]
and similarly for the $b$-ary sets.
Results from~\cite{buglau} 
indicate that $w=1/m$ is a metrical threshold value as
in the homgeneous case, so considering $\mathcal{S}_{m,\ut}$ is plausible.

\subsection{Organisation of the paper} The remainder of the paper is 
organized as follows. In \S~\ref{se2}, we establish lower bounds
for the packing dimension of the set of inhomogeneously ($b$-ary) singular vectors,
and also address the case of multiple
inhomogenities $\ut$ simultaneously. 
In \S~\ref{se3} we provide independent
results on sumsets of singular vectors, as well as topological
results. In particular we show that $\mathcal{S}_m$ is not a comeagre set.
In \S~\ref{se4} we deal with Cartesian products of missing digit sets (the Cantor middle third set is a special case), classical fractals. In the homogeneous case $\ut=\underline{0}$, we show that totally irrational vectors with
arbitrary order of singularity $w<1$ exist in these fractals 
and again provide lower bounds for their packing and Hausdorff dimension.

\section{Metrical results for inhomogeneous, $b$-ary singular vectors in $\Rm$} \label{se2}

\subsection{Metrical claims in $\Rm$}  \label{s2.1}

Our first new result gives lower bounds for the packing dimension
of the set of inhomogeneously $b$-ary singular vectors in $\Rd$.

\begin{theorem} \label{H}
	Let $m\geq 1$, $b\geq 2$ integers and $\ut\in\Rm$. Then for 
	any $w\in[0,1)$ we have
	\begin{equation} \label{eq:unfrei}
	\dim_{P}(\mathcal{S}_{m,\ut}^{\ast}(w)) \geq \dim_{P}(\mathcal{S}_{m,\ut}^{(b)\ast}(w)) \geq (1-w)m>0.
	\end{equation}
	In particular
	\begin{equation} \label{eq:frei}
	\dim_{P}(\mathcal{S}_{m,\ut}^{\ast}) \geq \dim_{P}(\mathcal{S}_{m,\ut}^{(b)\ast}) \geq m-1.
	\end{equation}
\end{theorem}

\begin{remark} \label{ureh}
	If $\ut=\underline{0}$,
	then \eqref{eq:unfrei}, \eqref{eq:frei} for
	$\mathcal{S}_{m}^{\ast}(w)$ and 
	$\mathcal{S}_{m}^{\ast}$ instead
are superseeded by \eqref{eq:dfsuoben}, and regarding the sets $\mathcal{S}_{m}(w)$ and $\mathcal{S}_{m}$ also the bound in
 \eqref{eq:tr} is larger. Moreover, the Hausdorff (thus packing) dimension of the non-$b$-ary sets $\mathcal{S}_{m,\ut}^{\ast}(w)$
 is in fact full when $w<1/m$ by Theorem~\ref{blau}. So the non-$b$-ary claims are of interest for $\ut\neq \underline{0}$ and $w\geq 1/m$.
\end{remark}

\begin{remark}
	When combining our method with the one of~\cite[Theorem~2.3]{arxiv3}
	based on~\cite[Example~4.6]{falconer},
	we also obtain a positive lower bound for the Hausdorff dimension, which is however considerably weaker. A calculation yields $\dim_H(\mathcal{S}_{m,\ut}^{(b)\ast}(w))\geq \left(\frac{1-w}{1+w}\right)^2$.
\end{remark}

Note that when $\ut=\underline{0}$, we cannot have anything larger than $m-1+\frac{1}{m+1}$
in \eqref{eq:frei} by Theorem~\ref{DFSU}, so up to $O(m^{-1})$ 
it is optimal. 
We stress that the setup in Theorem~\ref{H} is more general 
than Theorem~\ref{DFSU} in the sense
that it holds in inhomogeneous and $b$-ary setting. In particular, we
should expect strict inequality $\dim_{P}(\mathcal{S}_{m}^{(b)\ast})<\dim_{P}(\mathcal{S}_{m}^{\ast})$, as is also
suggested by \eqref{eq:jarnik} below. It is thus
natural that our new estimate \eqref{eq:unfrei}
is weaker than \eqref{eq:dfsuoben}.
We further notice that for $m=2$ the bound is identical
with the upper bound for the Hausdorff dimension in \eqref{eq:bchch}. 

Our proof of Theorem~\ref{H} and in fact all claims 
below are surprisingly easy and completely
unrelated to any previous approach, as for example in~\cite{dfsu2}.
We remark that our method a fortiori 
readily extends to general inhomogeneous systems of $m$ linear
forms in $n$ variables, 
however the results turn out most interesting 
in our setting $n=1$. Moreover, for $n>1$ the exact
value of the packing dimensions of levelsets in the homogeneous
problem are known~\cite[Theorem~3.8]{dfsu1}.

We also obtain a slightly weaker lower bound for the Hausdorff dimension. 

\begin{theorem}  \label{ax3}
	With the assumption and notation of Theorem~\ref{H}, we have
	\[
	\dim_H(\mathcal{S}_{m,\ut}^{\ast}(w))\geq \dim_H(\mathcal{S}_{m,\ut}^{(b)\ast}(w))\geq m\left(\frac{1-w}{1+w}\right)^2.
	\]
	In particular
	\[
	\dim_{H}(\mathcal{S}_{m,\ut}^{\ast}) \geq \dim_{H}(\mathcal{S}_{m,\ut}^{(b)\ast}) \geq m\left(1-\frac{2}{m+1}\right)^2.
	\]
\end{theorem}

Note that the latter bound is of order $m-4+O(m^{-1})$. See
also~\cite[Theorem~2.3]{arxiv3}.

\subsection{Singularity with respect to several $\ut$ simultaneously} \label{se2.2}
We now investigate sets in $\Rm$ that belong to $\mathcal{S}_{m,\ut}$
simultaneously for several $\ut\in\Rm$. To our knowledge, this topic
has not been studied before, see however the related remarks 
on~\cite{bglasgow, ts, kimliao} in the last paragraph
of this section. Denote by 
$e=2.71\ldots$ Euler's number. 

\begin{theorem}  \label{simu}
	Let $m\geq 1$, $k\geq 1, b\geq 2$ be integers
	and $\Theta=\{ \ut_1,\ldots,\ut_k \}$
	with $\ut_s\in\Rm, 1\leq s\leq k$ arbitrary. For any $w\in[0,1)$
	we have
	\begin{equation}  \label{eq:bound}
\dim_P(\bigcap_{\ut\in \Theta} \mathcal{S}_{m,\ut}^{\ast}(w))\geq
\dim_P(\bigcap_{\ut\in \Theta} \mathcal{S}_{m,\ut}^{(b)\ast}(w))\geq (1-kew)m.
	\end{equation}
	In particular, if $k<m/e$,
	then for any $\Theta=\{ \ut_1,\ldots,\ut_k \}$ we have
	\[
\dim_P(\bigcap_{\ut\in \Theta} \mathcal{S}_{m,\ut}^{\ast})\geq \dim_P(\bigcap_{\ut\in \Theta} \mathcal{S}_{m,\ut}^{(b)\ast})>0. 
\]
\end{theorem} 

Again for $w<1/m$ the claim for the non-$b$-ary sets is implied by Theorem~\ref{blau} since
full measure is preserved under countable intersections.
When $k=1$, compared to Theorem~\ref{H} our bound \eqref{eq:bound} is poorer due to the appearance of the factor $e$. However,
as the proof will show, this factor can be improved for 
given $k$, and rather significantly when $k$
is small.
We further remark that for any given element $\Theta$
of the power set of $\Rm$, we have $\cap_{\ut\in\Theta} \mathcal{S}_{m,\ut}(w)=\cap_{\ut\in\scp{\Theta}}\mathcal{S}_{m,\ut}(w)$
where $\scp{\Theta}$ denotes
the $\mathbb{Z}$-module in $\Rm$ spanned by $\Theta$ and the canonical
base vectors $\underline{e}_i=(0,\ldots,0,1,0\ldots,0)$, $1\leq i\leq m$. 

Regarding $b$-ary approximation, we can present a reverse result.
We restrict to $m=1$ and $k=2$ for simplicity, it can however be generalized.

\begin{theorem}  \label{reverse}
	Let $\theta_1= \sum_{N\geq 1} 3^{-N!}, \theta_2=2\theta_1$.
	Then for any $\xi\in\mathbb{R}$ we have
	\[
	\wo_{1,\theta_1}^{(3)}(\xi)+  \wo_{1,\theta_2}^{(3)}(\xi) \leq 1.
	\]
	In particular, for any $w>1/2$, we have
	\[
	\mathcal{S}_{1,\theta_1}^{(3)}(w)\cap \mathcal{S}_{1,\theta_2}^{(3)}(w)= \emptyset.
	\] 
\end{theorem}

The proof admits plenty of freedom in the choice of $\theta_i$ in the theorem,
and there is no significance of the base $3$ either.
We next want to establish below a new result on ordinary, 
inhomogeneous approximation.
Complementing the sets $\mathcal{S}_m(w)$ and $\mathcal{S}_m^{(b)}(w)$,  regarding ordinary approximation we define similarly the nested sets
\[
\mathcal{W}_{m,\ut}(w)=\{ \ux\in\Rm: \om_{m,\ut}(\ux)\geq w \},\qquad 
\mathcal{W}_{m,\ut}^{(b)}(w)=\{ \ux\in\Rm: \om_{m,\ut}^{(b)}(\ux)\geq w \},
\]
for $w\in[0,\infty]$. We may introduce variants
where we restrict to $\Rr$ or $\Rd$ as we did for singular vectors, however the distinction is of no relevance here. For any parameter $w$, we have
the obvious inclusions $\mathcal{W}_{m,\ut}^{(b)}(w)\subseteq \mathcal{W}_{m,\ut}(w)$ and
\[
\mathcal{S}_{m,\ut}(w)\subseteq \mathcal{W}_{m,\ut}(w), \qquad 
\mathcal{S}_{m,\ut}^{(b)}(w)\subseteq \mathcal{W}_{m,\ut}^{(b)}(w). 
\]
See Marnat and Moshchevitin~\cite{mamo} for refinements of the left inclusion when $\ut=\underline{0}$, in which case we get the classical sets $\mathcal{W}_m(w)=\mathcal{W}_{m,\underline{0} }(w)$ and $\mathcal{W}_m^{(b)}(w)=\mathcal{W}_{m,\underline{0} }^{(b)}(w)$ given as
\[
\mathcal{W}_m(w)= \{ \ux\in\Rm: \omega_m(\ux)\geq w \}, \qquad \mathcal{W}_m^{(b)}(w)= \{ \ux\in\Rm: \omega_m^{(b)}(\ux)\geq w \}.
\]
Dirichlet's Theorem implies
$\mathcal{W}_m(1/m)=\Rm$.  Moreover
\begin{equation} \label{eq:jarnik}
\dim_H(\mathcal{W}_m^{(b)}(w))=\frac{m}{w+1},\; w\in[0,\infty], \qquad \dim_H(\mathcal{W}_m(w))= \frac{m+1}{w+1}, \; w\in[1/m,\infty],
\end{equation}
where the left identity follows from Borosh and Fraenkel~\cite{bf72}, 
and the right identity is a famous result due to Jarn\'ik~\cite{jarnik}. The packing dimension of both sets is full,
a stronger result can be found in~\cite[Corollary~4]{arxiv}. 
For the larger
sets $\mathcal{W}_m(w)$, alternatively this is consequence of independent
results by Marnat~\cite{marnat} deduced
from the variational principle established in~\cite{dfsu1, dfsu2}.
Theorem~\ref{heep} below extends these claims and shows that
for ordinary approximation, 
we get considerably stronger metrical results than in Theorem~\ref{simu}.

\begin{theorem} \label{heep}
	Let $m\geq 1, b\geq 2$ integers and $\Theta=\{ \ut_1,\ldots \}$ be countable with $\ut_i\in\Rm$. Then 
	\[
	\dim_P(\bigcap_{\ut\in\Theta} \mathcal{W}_{m,\ut}(\infty))= \dim_P(\bigcap_{\ut\in\Theta} \mathcal{W}_{m,\ut}^{(b)}(\infty))=m. 
	\] 
\end{theorem}

Combining our proof below with the argument in Remark~\ref{hirsch}
below, it follows that any countable intersection over sets
$\widetilde{W}_{m,\ut_i}^{(b)}(w_i):=
\{ \ux\in\Rm: \om_{m,\ut_i}^{(b)}(\ux)=w_i \}$ 
of exact $b$-ary order of ordinary, inhomogeneous 
approximation $w_i\in[0,\infty]$, has full
packing dimension. However, this claim for the 
accordingly defined non-$b$-ary sets $\widetilde{W}_{m,\ut_i}(w_i)$ is not obvious. 

Notice that the sets $\Theta$ in Theorems~\ref{simu},~\ref{heep} 
are countable.
Regarding uncountable intersections, when $m=1$, some
information can be obtained from
the papers by Bugeaud~\cite{bglasgow} or Troubetzkoy and Schmeling~\cite{ts} for ordinary approximation,
and Kim and Liao~\cite{kimliao}
for uniform approximation. A particular consequence of~\cite{kimliao} is that for any given $w<\infty$, there is certain $\Theta\subseteq \mathbb{R}$ of positive Hausdorff
dimension, so that the intersection over $\theta\in\Theta$
of the sets $\mathcal{S}_{1,\theta}(w)$
is non-empty.
%
%

\section{Topological properties of singular vectors}  \label{se3}

\subsection{Sumsets of singular vectors}
Our method for the proofs of claims in \S~\ref{se2}
is based on investigation of sumsets.
We state some related result directly addressing sumsets, 
obtained from a similar method. Denote $A+B=\{a+b: a\in A, b\in B\}$ for $A,B\subseteq \Rm$.

\begin{theorem} \label{t3}
	Let  $m\geq 1, b\geq 2$ be integers and $\ut_0,\ut_1\in\Rm$.
	If $w\in(0,1)$ and 
	\begin{equation}  \label{eq:idne}
	w^{\prime}= (1-\sqrt{w})^2=  w+1-2\sqrt{w},
	\end{equation}
	then for any pair of real numbers $(w_0, w_1)$ satisfying
	$0\leq w_0<w$ and $0\leq w_1<w^{\prime}$, the identity of sets
	\begin{equation}  \label{eq:G0}
	\mathcal{S}_{m,\ut_0}^{(b)}(w_0) + \mathcal{S}_{m,\ut_1}^{(b)}(w_1)=
	\mathcal{S}_{m,\ut_0}(w_0) + \mathcal{S}_{m,\ut_1}(w_1)= \Rm
	\end{equation}
 holds. Moreover the sets
   \begin{equation}  \label{eq:fullm}
   \mathcal{S}_{m,\ut_0}^{(b)\ast}(w_0) + \mathcal{S}_{m,\ut_1}^{(b)\ast}(w_1), \qquad
   \mathcal{S}_{m,\ut_0}^{\ast}(w_0) + \mathcal{S}_{m,\ut_1}^{\ast}(w_1)
   \end{equation}
   have full $m$-dimensional Lebesgue measure.
\end{theorem}

\begin{remark}  \label{rmark}
	We believe that the set identities in fact hold for the sets with $\ast$ as well, but technical obstructions on small sets occur in our proof. 
\end{remark}

We want to highlight the special, symmetric case $w=1/4$.

\begin{corollary}  \label{mko}
	With notation and assumptions of Theorem~\ref{t3}, for any $w<1/4$ we have
	\begin{equation}  \label{eq:B}
	\mathcal{S}_{m,\ut_0}^{(b)}(w) + \mathcal{S}_{m,\ut_1}^{(b)}(w)=
	\mathcal{S}_{m,\ut_0}(w) + \mathcal{S}_{m,\ut_1}(w)=\Rm,
	\end{equation}
	and the according sumsets with $\ast$ throughout have full $m$-dimensional Lebesgue measure.
\end{corollary}


The claim is most 
interesting for $m\geq 5$, otherwise 
the sets $\mathcal{S}_{m,\ut_i}(w)$ with $w<1/4$ have full $m$-dimensional Lebesgue measure by Theorem~\ref{blau}, consequently the 
implication for the (larger) set 
$\mathcal{S}_{m,\ut_0}(w) + \mathcal{S}_{m,\ut_1}(w)$ follows from
an easy argument.
From \eqref{eq:B} with $\ut_0=\ut_1=\ut$ and \eqref{eq:AB} below 
we could deduce the lower bound $m/2$ for
the packing dimension of $\mathcal{S}_{m,\ut}^{(b)}(w)$ for
any $\ut\in\Rm$ and $w<\frac{1}{4}$, however this is superseeded by Theorem~\ref{H}.
From the full measure claim \eqref{eq:fullm}, we infer some information
on the Hausdorff dimension of the homogeneous sets 
$\mathcal{S}_m^{\ast}(w)$ for certain $w$ close to $1$. Corollary~\ref{coro} 
below may be regarded
complementary to~\cite[Theorem~1.7]{dfsu1}, in fact it shows that
the interval for the validity of its left estimate
has length $\ll 1/\sqrt{m}$.

\begin{corollary}  \label{coro}
	For $m\geq 2$ and any $w<(\sqrt{m}-1)^2/m$, we have
	\[
	\dim_H(\mathcal{S}_{m}^{\ast}(w)) \geq\dim_H(\mathcal{S}_{m}^{(b)\ast}(w)) \geq 1-\frac{1}{m+1}.
	\]
\end{corollary}

We prove the corollary in \S~\ref{pko}.
The analogous claim for $\mathcal{S}_m(w)$ instead would be 
trivially implied by \eqref{eq:tr}.
Similar to \eqref{eq:B},
Corollary~\ref{coro} is most interesting for $m\geq 5$, otherwise the
parameter bound does not exceed the critical value $1/m$.  
We next notice that $1/4$ in Corollary~\ref{mko}
cannot be replaced by $1$, at least in the
homogeneous case $\ut_0=\ut_1=\underline{0}$. This is an easy 
corollary of results in~\cite{dfsu1}.

\begin{theorem} \label{t4}
	For $m\geq 2$ and $\epsilon>0$, 
	for $w=w(m,\epsilon)<1$ close enough to $1$ we have
	\[
	\dim_H(\mathcal{S}_m(w) + \mathcal{S}_m(w))= 2, \qquad
	\dim_H(\mathcal{S}_m^{\ast}(w) + \mathcal{S}_m^{\ast}(w))\leq 1+\epsilon.
	\]
\end{theorem}

The reverse claims in Corollary~\ref{mko} and Theorem~\ref{t4} 
motivate to determine the threshold exponents where the 
topological behavior of the sumsets change.

\begin{problem}
	Determine $\mathcal{T}_m$ resp. $\mathcal{T}_{m}^{\ast}$ defined as
	the supremum of $w$ so that the sumset
	$\mathcal{S}_m(w) + \mathcal{S}_m(w)$ resp. $\mathcal{S}_{m}^{\ast}(w) + \mathcal{S}_{m}^{\ast}(w)$ equals $\Rm$.
\end{problem}

It is clear that $\mathcal{T}_1=\mathcal{T}_{1}^{\ast}=1$.
By considering two rational lines in $\mathbb{R}^2$ as in the proof
of Theorem~\ref{t4} below,
we still easily get $\mathcal{T}_2=1$. 
However, we can deduce from the results above Dirichlet's Theorem
that
$\mathcal{T}_m\in [ \max\{ \frac{1}{4}, \frac{1}{m}\},1)$ when $m\geq 3$, 
and moreover $\mathcal{T}_m^{\ast}\in [\frac{1}{m},1)$ when $m\geq 2$ (conjecturally $1/4$ is also a lower bound). We expect
analogous results
for accordingly defined inhomogeneous suprema $\mathcal{T}_{m,\ut}, \mathcal{T}_{m,\ut}^{\ast}$.

\subsection{On the Baire category of singular vectors}  \label{top}
We next turn towards topological claims on the set
of singular vectors. Recall a set is called comeagre (or residual) 
if it is the intersection
of countably many open dense sets, or equivalently its complement
is of first Baire category. The set of Liouville numbers $\mathcal{W}_1^{\ast}(\infty)=W_1(\infty)\cap \mathcal{R}_1=W_1(\infty)\setminus \mathbb{Q}$ is well-known to be comeagre~\cite{oxtoby}. By a very similar
argument this extends to Liouville vectors $\mathcal{W}_m^{\ast}(\infty)$, and further
clearly to $\mathcal{W}_m(\infty)$ as the difference set
is contained in a countable union of affine hyperplanes. 
We show that this is not true for $\mathcal{S}_m$. This relies
on our next very general observation
linking toplogy and measure theory.

\begin{theorem}  \label{T2}
	If $A\subseteq \Rm$ is comeagre, then $\dim_P(A)=m$.
\end{theorem}

To our knowledge this has not been noticed before.
The claim is far from being true for the Hausdorff dimension, as the set $\mathcal{W}_m(\infty)\subseteq \Rm$ (and $\mathcal{W}_m^{(b)}(\infty)$ as well) provides a counterexample of Hausdorff dimension $0$ in view 
of \eqref{eq:jarnik}. 
As
the inhomogeneous sets $\mathcal{W}_{m,\ut}(\infty)$ and $\mathcal{W}_{m,\ut}^{(b)}(\infty)$ 
are comeagre as well by the same argument as for $\ut=\underline{0}$,
Theorem~\ref{T2} further leads to an alternative, unconstructive proof of Theorem~\ref{heep}. Moreover, from Theorem~\ref{DFSU} we immediately infer topological information on the sets $\mathcal{S}_m$.

\begin{corollary} \label{T1}
	For any $m\geq 1$, the set $\mathcal{S}_m$ is not comeagre in $\Rm$.
\end{corollary}

It can be shown that $\mathcal{S}_m$ has non-empty intersection with the comeagre set
of Liouville vectors $\mathcal{W}_m^{\ast}(\infty)$ (follows 
for example from~\cite[Theorem 2.5]{ichostrava}),
however Corollary~\ref{T1} shows that the intersection is small 
in a topological sense.
We strongly believe that the claim extends to the inhomogeneous sets $\mathcal{S}_{m,\ut}$
as well, however from our results above we cannot exclude
that it has full packing dimension for certain $\ut$. 
We wonder if the Corollary~\ref{T1} remains true for the larger set
$Di_m\subseteq \Rm$ of Dirichlet improvable vectors, 
defined as the set of $\ux$ for which \eqref{eq:1} is soluble
for some $c<1$ and all large $Q$ (we let $\Vert.\Vert$ be 
the maximum norm here). We state 
an open problem.

\begin{problem}
	Let $m\geq 2$. Decide whether the set $\mathcal{S}_m$ is meagre or not. What can we say topologically about $Di_m$?
\end{problem}

For $m=1$ we have $\mathcal{S}_1=\mathbb{Q}$ and $Di_1=\mathcal{S}_1\cup Bad_1$, where $Bad_m$ denotes the set of badly approximable vectors,
consisting of $\ux\in\Rm$ for which \eqref{eq:1} admits no 
integer solution for some $c>0$ and any $Q>1$.
Since $Bad_m\subseteq W_m(\infty)^c$ are meagre, indeed so are
$Di_1$ and $\mathcal{S}_1$.
Generally $Di_m\supseteq \mathcal{S}_m\cup Bad_m$ holds~\cite[Theorem~2]{ds},
however for $m\geq 2$, the set $Di_m\setminus (Bad_m\cup \mathcal{S}_m)$
is non-empty~\cite{beretc},
and conjecturally has full Hausdorff dimension.
Hence no conclusion from Theorem~\ref{T2} can be drawn. 

\subsection{Sets with remarkable sum and product sets}
We want to add another 
result that is partly motivated
by the proof of Theorem~\ref{T2}, however some claims 
heavily rely on
recent results from~\cite{art}. For $A,B\subseteq \Rm$
and $\circ\in\{ +,-,\cdot,/\}$ we denote by
\begin{align*}
A\circ B=\{ (a_1\circ b_1,\ldots,a_m\circ b_m):\;\; (a_1,\ldots,a_m)\in A,\; (b_1,\ldots,b_m)\in B\},
\end{align*} 
their coordinatewise sum, difference, product and quotient (where we
restrict to $b_i\neq 0$ in the last case).

\begin{theorem}  \label{C3}
	Let $m\geq 1$. There exist sets $A,B\subseteq \Rm$ that 
	for any operation $\circ\in\{ +,-,\cdot,/\}$ satisfy
	\begin{equation}  \label{eq:sumse}
	A\circ A=B\circ B=\Rm, \qquad \dim_H(A\circ B)<m.
	\end{equation}
	For $m\geq 5$, claim \eqref{eq:sumse} for $\circ\in\{ +,-\}$
	holds in particular for $A=\mathcal{W}_m(\infty)$ 
	and $B=\mathcal{S}_m$.
\end{theorem}

While it is not hard to construct $A,B$ with $A\pm A=B\pm B=\Rm$ and
$A\pm B\neq \Rm$, it suffices to take $A=\mathcal{W}_m^{\ast}(\infty)$ 
the set of Liouville
numbers and $B=A^c$ its complement, 
we are unaware of any previous notice of
the stronger claim \eqref{eq:sumse} for $\circ=+$ in Corollary~\ref{C3}. 
We want to formulate a problem regarding refinements of Theorem~\ref{C3}.

\begin{problem}
	Do there exist sets $A,B\subseteq \Rm$ for which
	\begin{equation}  \label{eq:Stern}
	A+A=B+B=\Rm, \qquad \dim_H(A+B)=0 
	\end{equation}
	or
		\begin{equation}  \label{eq:Katz}
	A+A=B+B=\Rm, \qquad \dim_P(A+B)<n 
	\end{equation}
	holds?
\end{problem}

We cannot take packing dimension in \eqref{eq:Stern}. 
If $A+A=\Rm$
then $\dim_P(A)\geq m/2$ in view of 
$\dim_P(A+C)\leq \dim_P(A\times C)\leq \dim_P(A)+\dim_P(C)$ 
for all $A,C\subseteq\Rm$ (see~\cite{falconer}), 
applied to $C=A$. Thus since clearly $B\neq \emptyset$
we also have $\dim_P(A+B)\geq \dim_P(A)\geq m/2$.
We cannot have both claims \eqref{eq:Stern}, \eqref{eq:Katz} simultaneously either, since by \eqref{eq:tricot}
\begin{align*}
m&=\dim_H(\Rm)=\dim_H(A+A)=\dim_H((A+A)+B)=\dim_H(A+(A+B))\\ &\leq \dim_H(A+B)+\dim_P(A)\leq \dim_H(A+B)+\dim_P(A+B).
\end{align*}

\section{Singular vectors in fractal sets in the homogeneous setting} \label{se4}

Singular vectors on fractals have recently received much attention. 
We only mention a few important results.
Kleinbock and Weiss showed in~\cite{KW} 
that for a wide class of fractals $K\subseteq \Rm$, the
set $\mathcal{S}_m\cap K$ is a nullset with respect to the natural
measure on $K$. On the other hand,
Kleinblock, Moshchevitin and Weiss~\cite{kmw} established that for a large class of fractals $K\subseteq \Rm$, 
the set of singular vectors in $K$ is non-empty, 
more precisely $\mathcal{S}_m^{\ast}(w)\cap K$ 
is uncountable for any $w<\frac{1}{m-1}$.
For $K$ the attractor of an iterated function system (upon some
regularity conditions),
upper bounds for the Hausdorff dimension of $\mathcal{S}_m\cap K$ were  
obtained by Khalil~\cite{khalil}, contributing to a problem
posed by Bugeaud, Cheung and Chevallier~\cite[Problem~6]{bcc} who were
interested in $K\subseteq \mathbb{R}^2$ the Cartesian product of two copies of the Cantor 
middle third set. We complement these results by lower bounds for the packing dimension of $\mathcal{S}_m^{\ast}\cap K$, for $K\subseteq \Rm$ certain
Cartesian products of classical missing digit Cantor sets. This generalizes the setting raised in~\cite{bcc} and is
also explicitly addressed in~\cite{khalil}. We must
restrict ourselves to the homogeneous case $\ut=\underline{0}$ here. 

For $b\geq 2$ an integer and $W\subseteq \{0,1,\ldots,b-1\}$ with $|W|\geq 2$,
let $C_{b,W}\subseteq \mathbb{R}$ be the 
missing digit Cantor set of real numbers that admit a 
base $b$ representation
\[
\xi= \sum_{i\geq 1} a_ib^{-i}, \qquad a_i\in W.
\]
For the sets $C_{b,W}$ the Hausdorff and packing dimensions coincide
and take the value
\begin{equation}  \label{eq:haupack}
\dim_H(C_{b,W})=\dim_P(C_{b,W})=\frac{\log |W|}{\log b},
\end{equation}
see Falconer~\cite{falconer}. 
Below we just write $\dim(A)$ without index 
for any set $A$ with this property, meaning either dimension.
We consider Cartesian products
\begin{equation}  \label{eq:wiro}
K=C_{b,W_1}\times \cdots\times C_{b,W_m}
\end{equation}
with arbitrary $W_i\subseteq \{0,1,\ldots,b-1\}$, $|W_i|\geq 2$, and uniform $b$.
A well-known consequence of \eqref{eq:haupack} reads
\begin{equation}  \label{eq:kpro}
\dim(K)= \sum_{i=1}^{m} \dim(C_{b,W})=\frac{\sum_{i=1}^{m} \log |W_i|}{\log b}.
\end{equation}
Indeed, this can be deduced
from
\begin{equation}  \label{eq:kproS}
\dim_H(A) + \dim_H(B)\leq \dim_H(A\times B)\leq \dim_P(A\times B)\leq \dim_P(A) + \dim_P(B)
\end{equation}
valid for any Euclidean sets $A,B$
(see Falconer~\cite{falconer}). For brevity let us
write 
\[
\mathcal{S}_K^{\ast}(w)= \mathcal{S}_m^{\ast}(w)\cap K,\quad \mathcal{S}_K^{(b)\ast}(w)= \mathcal{S}_m^{(b)\ast}(w)\cap K,  \qquad\qquad w\in[0,1),
\]
and
\[
\mathcal{S}_K(w)= \mathcal{S}_m(w)\cap K,\quad \mathcal{S}_K^{(b)}(w)= \mathcal{S}_m^{(b)}(w)\cap K,  \qquad\qquad w\in[0,1).
\]
The set $C_{b,W}$ supports a natural probability measure $\mu_{b,W}$,
essentially the restriction of the $(\log |W|/\log b)$-dimensional 
Hausdorff measure to $C_{b,W}$,
and consequently the product measure $\mu=\mu_{b,W_1}\times \cdots\times \mu_{b,W_m}$
is the natural measure on $K$. Variants of Theorem~\ref{H} and Theorem~\ref{t3} for homogeneous approximation in $K$ read as follows. 

\begin{theorem}  \label{cantor}
	Let $m\geq 1$, $b\geq 2$ be integers and $W_i$, $K$ and $\mathcal{S}_K(w), \mathcal{S}_K^{(b)}(w), \mathcal{S}_K^{\ast}(w)$ and $\mathcal{S}_K^{{(b)}\ast}(w)$ as above. Then the following two claims hold:
	\begin{itemize}
		\item[(i)] Assume $0\in W_i$ for $1\leq i\leq m$. Then for any $w\in[0,1)$ we have
		\begin{align} \label{eq:csr}
		\dim_P(\mathcal{S}_K^{\ast}(w))\geq \dim_P(\mathcal{S}_K^{(b)\ast}(w))
		\geq (1-w)\dim(K)=
		(1-w)\frac{\sum_{i=1}^{m} \log |W_i|}{\log b}>0.    
		\end{align}
		In particular
		\begin{equation}  \label{eq:habicht}
		\dim_P(\mathcal{S}_K^{\ast}(w))\geq \dim_P(\mathcal{S}_K^{(b)\ast}(w))\geq 
		\left(1-\frac{1}{m}\right)\dim(K)=\left(1-\frac{1}{m}\right)\frac{\sum_{i=1}^{m} \log |W_i|}{\log b}.
		\end{equation}
		For arbitrary $W_i$, for the non-$b$-ary sets still we have
		the same estimate
		\begin{equation}  \label{eq:nonedda}
		\dim_P(\mathcal{S}_K^{\ast}(w))\geq (1-w)\dim(K)=
		(1-w)\frac{\sum_{i=1}^{m} \log |W_i|}{\log b}>0.  
		\end{equation}
		\item[(ii)] Assume $0\in W_i$ for $1\leq i\leq m$.  Let $w\in(0,1)$ and $w^{\prime}=(1-\sqrt{w})^2$ 
		as in \eqref{eq:idne}. Then for any $w_0<w$ and $w_1<w^{\prime}$ the sets
		\[
		(\mathcal{S}_K^{\ast}(w_0) + \mathcal{S}_K^{\ast}(w_1))\cap K, \qquad 
		(\mathcal{S}_K^{(b)\ast}(w_0) + \mathcal{S}_K^{(b)\ast}(w_1))\cap K
		\]
		have full $\mu$-measure and we have the inclusion 
		\[
		\mathcal{S}_K(w_0) + \mathcal{S}_K(w_1)\supseteq
		\mathcal{S}_K^{(b)}(w_0) + \mathcal{S}_K^{(b)}(w_1)\supseteq K. 
		\]
	\end{itemize}
\end{theorem}

\begin{remark}
	Similar to Remark~\ref{rmark},
	we believe that the stronger inclusion property in (ii) 
	holds for the sets restricted to $\Rr$ (denoted with asterisks $\ast$) as well. Similar to Remark~\ref{ureh}, without
	restriction to $\Rr$ in (i), the non-$b$-ary claims are implied by the according estimate \eqref{eq:tr} up to a factor $\dim(K)$ on the right hand side, obtained by the same argument. 
\end{remark}

\begin{remark}
	The claim in particular implies the existence of
	vectors in $K\cap \Rr$ singular to any order $w<1$, for $K$
	as in \eqref{eq:wiro}.
	In fact, we can also obtain $\wo_m(\ux)=1$ by a minor twist of the proof, however then we cannot guarantee positive packing dimension. This topic is discussed in~\cite[\S~1.6]{kmw}, see the first paragraph of this section.
	We point out that a construction
	of $\ux\in K\cap \Rr$ with any prescribed exponent $\wo_m(\ux)=\wo_m^{(b)}(\ux)\in[1/m,1]$, is implicitly contained
	in~\cite[Theorem~2.5]{ichostrava}, upon
	a suitable choice of involved $q_{i,j}$ and $\eta_i$ (see also~\cite[Corollary~2.11]{ichostrava} when $\wo_m(\ux)\in(1/2,1]$). However, no metrical claims are given in~\cite{ichostrava}.
	The problem 
	if $\mathcal{S}_m^{\ast}(w)\cap K$ 
	is non-empty for $w\geq \frac{1}{m-1}$ for $K$ more generally a product of perfect sets in which $\mathbb{Q}$ is dense treated in~\cite{kmw} remains open as well. 
\end{remark}

The claims simplify to the homogeneous case of Theorem~\ref{H} and 
Theorem~\ref{t3} if $W=\{0,1,\ldots,b-1\}$. 
%
We enclose the special case of the product of two copies 
of the Cantor middle
third set, as proposed in~\cite[Problem~6]{bcc}, and further
treat $\mathbb{R}\times C_{3,\{0,2\}}$.

\begin{corollary}  \label{uh}
	Let $K=C_{3,\{0,2\}}\times C_{3,\{0,2\}}$. Then
	\[
	\dim_P(\mathcal{S}_K^{\ast})\geq \dim_P(\mathcal{S}_K^{(b)\ast})\geq \frac{\dim(K)}{2}= \frac{\log 2 }{\log 3}.
	\]
	For $\tilde{K}= \mathbb{R}\times C_{3,\{0,2\}}$ we have
		\[
	\dim_P(\mathcal{S}_{\tilde{K}}^{\ast})\geq \dim_P(\mathcal{S}_{\tilde{K}}^{(b)\ast})\geq \frac{\dim(\tilde{K})}{2}=\frac{\log 2}{2\log 3}+ \frac{1}{2}.
	\]
\end{corollary}

We also obtain a generalization of
Theorem~\ref{ax3} on Hausdorff dimension of singular vectors in the fractals, however the result is not as clean.

\begin{theorem}  \label{khr}
	Let $K=\prod C_{b,W_i}$ as above and assume $0\in W_i$ for all $1\leq i\leq m$. Write $d_i= \dim(C_{b,W_i})=\log |W_i|/\log b$ for $1\leq i\leq m$. Then $\dim_H(\mathcal{S}_K^{(b)\ast})$ is bounded from below by
	the maximum of the function
	\[
	t\longmapsto \sum_{i=1}^{m} d_i \cdot 
	\frac{ -wt^2 + (w+1)t -1 }{ (w+1)(1-d_i) t^2 + ((w+2)d_i -1) t - d_i }
	\]
	over $t>1$.
	If $K=C_{3,\{0,2\}}\times C_{3,\{0,2\}}$, then
	\[
	\dim_H(\mathcal{S}_K^{\ast})\geq \dim_H(\mathcal{S}_K^{(3)\ast}) > 0.1255.
	\]
	If we drop the assumption $0\in W_i$, the estimates remain
	true for the non-$b$-ary sets $\mathcal{S}_K^{\ast}$.
\end{theorem}

If $d_1=d_2=\cdots=d_m=1$  it can be shown that the claim simplifies to Theorem~\ref{ax3}.
Theorem~\ref{khr} complements an estimate by Khalil~\cite{khalil} for the same set $K$ given by
\begin{equation}  \label{eq:kalil}
\dim_H(\mathcal{S}_K) \leq 
\frac{2}{3}\dim(K)= \frac{4\log 2}{3\log 3}. 
\end{equation}
Some variant of Corollary~\ref{coro} can also be inferred
from Theorem~\ref{cantor} for general $K$ as in \eqref{eq:wiro},
however its statement turns out to be considerably weaker 
(in many cases trivial) than the 
expected analogous claim 
\begin{equation}  \label{eq:FR}
\dim_H(\mathcal{S}_K^{\ast}(w))\geq \dim_H(\mathcal{S}_K^{(b)\ast}(w))\geq (1-\frac{1}{m+1})\dim(K), \qquad w<(\sqrt{m}-1)^2/m,
\end{equation}
since we lack an
analogue of Theorem~\ref{DFSU} for Cantor sets.
Lower bounds for $\dim_P(\mathcal{S}_K^{(b)\ast}(w))$
could be obtained when combining Tricot's estimate
\eqref{eq:tricot} below, claim (ii) in Theorem~\ref{cantor} and~\cite[Theorem~A]{khalil},
upon evaluating the expressions $\alpha_{\ell}(\mu)$ 
defined in~\cite{khalil},
where $\mu$ is the natural measure on $K$.
This computation is only done explicitly for $m=2$ and $K=C_{3,\{0,2\}}\times C_{3,\{0,2\}}$ in~\cite[Corollary~1.4]{khalil}, from which
\eqref{eq:kalil} was deduced.
However, only for $m\geq 5$ the induced bounds are not implied
by Theorem~\ref{blau}, so \eqref{eq:kalil} is not helpful in this matter. We further mention
that a general inhomogeneous, $b$-ary version of Theorems~\ref{cantor},~\ref{khr} cannot hold as soon as $W\subsetneq \{0,1,\ldots,b-1\}$, as then
$\{ b^{N}\ux-\underline{p}-\ut: \ux\in K, \underline{p}\in\mathbb{Z}^m, N\geq 1\}$ 
avoids a neighborhood of $\underline{0}$ 
for any $\ut\subseteq  [0,1)^m\setminus K$. 
It remains however unclear to us for the sets $\mathcal{S}_{m,\ut}^{\ast}(w)\cap K$ where
we do not restrict to $b$-ary setting.

\section{Proofs}

\subsection{Preparation for metrical results}  \label{sec2.1}

Our proofs below are based on a result by Tricot~\cite{tricot} on Cartesian products.

\begin{theorem}[Tricot]
	Any measurable sets $A\subseteq \mathbb{R}^{d_1},B\subseteq \R^{d_2}$ satisfy
	\begin{equation}  \label{eq:tricot}
	\dim_H(A\times B) \leq \dim_P(A) + \dim_H(B). 
	\end{equation}
\end{theorem}

See also Bishop and Peres~\cite{bishop}.
It is well-known that Lipschitz map $\Psi: \mathbb{R}^{d_1}\to \mathbb{R}^{d_2}$ satisfies
$\dim_H(\Psi(A))\leq \dim_H(A)$ for any $A\subseteq \mathbb{R}^{d_1}$,
see~\cite{falconer}. Applied to $d_1=2m, d_2=m$ and
the sum map $\Psi(A\times B)=A+B$ for $A,B\subseteq \Rm$ we see
$\dim_H(A+B) \leq \dim_H(A\times B)$. 
Combining with \eqref{eq:tricot} we obtain
\begin{equation} \label{eq:thor}
\dim_P(A) \geq \dim_H(A+B) - \dim_H(B)
\end{equation}
for any measurable $A,B\subseteq \Rm$. Moreover it is well-known (see again Falconer~\cite{falconer}) that
any measurable $A\subseteq \Rm$ satisfies
$\dim_H(A)\leq \dim_{P}(A)$.
Specializing to $A=B\subseteq \Rm$ and combining 
with above estimates, we see that
any $A\subseteq \Rm$ satisfies
\begin{equation}  \label{eq:AB}
\dim_P(A) \geq \frac{\dim_P(A) + \dim_H(A)}{2}\geq
\frac{\dim_H(A\times A)}{2}\geq \frac{\dim_H(A+A)}{2}.
\end{equation}

Due to our restriction to $\Rd$ in several claims, we
require the following lemma whose statement appears rather 
technical but which admits a straightforward proof. 
We formulate it in a general
Cantor set setting that allows us to apply it 
to Theorem~\ref{cantor} as well as any any other instance.
Recall the notation $\mu=\prod \mu_{b,W_i}$
for the natural measure on $K=\prod C_{b,W_i}$. 

\begin{lemma}  \label{daslemma}
	Let $\mathcal{A}\subseteq \mathbb{N}$ be infinite and denote its complement by
	$\mathcal{B}=\mathbb{N}\setminus \mathcal{A}$. Let 
	$b\geq 2$ an integer and $W_i\subseteq \{0,1,\ldots,b-1\}$, $|W_i|\geq 2$ for $1\leq i\leq m$,
	and $K$ as in \eqref{eq:wiro}. Given fixed sequences $(c_{i,n})_{n\geq 1}\in W_i^{\mathbb{N}}$ for $1\leq i\leq m$
	and any given partial maps
	\begin{align*}
	&\chi_i: \mathcal{B}\to W_i,    \\
	&\chi_i(n)= c_{i,n}
	\end{align*}
	writing $\ux=(\xi_1,\ldots,\xi_m)\in K$ consider the diagonal map
	\begin{align*}
	&\varphi: \prod_{i=1}^{m} (C_{b,W_i}\setminus \mathbb{Q}) \to K, \\
	&\varphi(\ux)=(\varphi_1(\xi_1),\ldots,\varphi_m(\xi_m)),
	\end{align*}
	with coordinate functions
	\[
	\varphi_i(\xi_i)= (0.e_{i,1}e_{i,2}\ldots )_b=\sum_{n\geq 1} \frac{e_{i,n}}{b^n}, \qquad 1\leq i\leq m, 
	\]
	where if $\xi_i=(0.d_{i,1}d_{i,2}\ldots)_b$, $d_{i,n}\in W_i$, then
	\[
	e_{i,n}= \begin{cases}
	d_{i,n}, \qquad n\in \mathcal{A},\\ c_{i,n},  \qquad n\in\mathcal{B}.
	\end{cases}
	\]
	Then, if $\mathcal{T}$ is any countable union of affine hyperplanes in $\Rm$, then for almost
	all $\ux\in K$ with respect to $\mu$ we have
	$\varphi(\ux)\notin \mathcal{T}$.
\end{lemma}

The map $\varphi$ changes the $b$-ary digits of $\xi_i\in C_{b,W_i}$ 
to prescribed digits on some set $\mathcal{B}$ and preserves the digits 
on the infinite complement set $\mathcal{A}$, in each coordinate.
The special case $W_i=\{0,1,\ldots,b-1\}$ for $1\leq i\leq m$
gives general $\ux\in K=\Rm$.
The restriciton to irrational $\xi_i$ is necessary only
if $\{0,b-1\}\subseteq W_i$, but if true then certain rational
numbers have two different base $b$
expansions, and then indeed our map would not be well-defined.
However, from a metrical point of view, neglecting countably many
affine hyperplanes does not make a difference for our claim.

\begin{proof}[Proof of Lemma~\ref{daslemma}]
	We proceed by induction on $m$. Let $m=1$. Then $\mathcal{T}$
	consists of countably many singleton points.
    Consider the set $\varphi^{-1}(\mathcal{T})$ defined
    on $\mathcal{T}\cap \varphi(K)$, 
	which consists of numbers in $K$ with the same digital representation
	of some element in $\mathcal{T}$ within intervals $\mathcal{A}$
	and arbitrary digits in $W=W_1$ at the remaining places in $\mathcal{B}$.
	Since $\mathcal{A}$ is infinite and choosing a fixed base
	$b$ digit within $W$ at some given position shrinks the $\mu_1=\mu_{b,W_1}$ measure of a set by a factor $|W|^{-1}<1$, the preimage of any singleton has $\mu_1$ measure $0$.
	Thus, as $\mathcal{T}$ is only countable, the set $\varphi^{-1}(\mathcal{T})$
	has $\mu_1$ measure $0$. Hence
	$K\setminus \varphi^{-1}(\mathcal{T})$ has full measure, in other
	words for some full measure set $F\subseteq K$ we have $\varphi(F)\subseteq \mathcal{T}^c$.
	The case $m=1$ follows.  
	
	We explain the induction step from $m-1$ to $m$. Recall that $\mu=\prod \mu_i$
	where $\mu_i:=\mu_{b,W_i}$ is the natural Cantor measure on $C_{b,W_i}\subseteq \mathbb{R}$.
	Since $\mathcal{T}$ consists of only countably many affine
	hyperplanes, there
	are in particular only countably many of them of the form 
	$H_{\xi_1}: x_1=\xi_1$, i.e. consisting of vectors $(\xi_1,x_2,\ldots,x_m)$
	with prescribed first coordinate $x_1$. Since affine hyperplanes
	have $\mu$ measure $0$ and sigma-additivity of measures,
	we can omit this countable union.
	Now for any other $\xi_1$ with $H_{\xi_1}$ not among
	the affine hyperplanes defining $\mathcal{T}$, the affine hyperplane
	$H_{\xi_1}$ intersects any affine hyperplane in $\mathcal{T}$ in 
	a lower dimensional affine space. 
	 We apply the induction hypothesis to any such $H_{\xi_1}$
	 to see that 
	 $\mu_{(m-1)}(\{ x\in H_{\xi_1}: \varphi(x)\in \mathcal{T}\})=0$, where
	 $\mu_{(m-1)}=\mu_2\times \cdots\times \mu_m$. 
	 Writing $\mu=\mu_1\times \mu_{(m-1)}$ and integrating 
	 the constant $0$ function over
	the first variable $\xi_1$ with respect to $\mu_1$, 
	we see that the for a full $\mu$ measure set $F$ we have $\varphi(F)\notin \mathcal{T}$, equivalent to the claim of the theorem.
\end{proof}

The proofs of our main results combine above observations,
in particular \eqref{eq:thor}, with 
ideas on sumsets similar as in~\cite{arxiv}. Before
we provide a short proof of Theorem~\ref{t4}.

\subsection{Proof of Theorem~\ref{t4}}

By combination of~\cite[Theorem~1.7, Theorem~2.9 and Remark~1.8]{dfsu1}, for $m\geq 2$ and $w<1$ sufficiently close to $1$,
we have
\[
\dim_H(\mathcal{S}_m^{\ast}(w))\leq \epsilon, \qquad \dim_P(\mathcal{S}_m^{\ast}(w))=1.
\]
Hence, by \eqref{eq:thor} we have
\begin{align*}
\dim_H(\mathcal{S}_m^{\ast}(w)+\mathcal{S}_m^{\ast}(w))
\leq 
\dim_P(\mathcal{S}_m^{\ast}(w))+\dim_H(\mathcal{S}_m^{\ast}(w))\leq 1+\epsilon,
\end{align*}
the right claim. For the left identity, the lower bound follows from considering the sum
of two rational non-parallel lines in $\Rm$
given by equations $x_j=r_jx_1$ resp.
$x_j=s_jx_1$ with $r_j, s_j\in\mathbb{Q}$, $2\leq j\leq m$, 
which becomes a two-dimensional subspace of $\Rm$ consisting of vectors
with $\wo_m(\ux)\in \{ 1,\infty\}$ depending on 
whether $\ux\in\mathbb{Q}^m$ or not. For
the upper bound, consider
for any rational affine subspace $\mathscr{V}\subseteq \Rm$
the projection of $\mathcal{S}_m(w)\cap \mathscr{V}$ to a space spanned 
by coordinate axes that induce a maximum $\mathbb{Q}$-linearly
independent set (i.e. all other coordinates of $\ux\in \mathscr{V}$ 
are $\mathbb{Q}$-linear combinations, inducing a rational subspace). 
The image consists of vectors in $\mathcal{S}_k^{\ast}(w)$ for some $1\leq k<m$
where the first claim applies. Clearly the projection is bi-Lipschitz for any fixed rational subspace and keeps thus the dimensions invariant.
By sigma-additivity of measures,
this
exhausts the Hausdorff/packing dimension of $\mathcal{S}_m(w)\cap \mathscr{V}$.
This argument shows that
\[
\dim_H(\mathcal{S}_m(w))\leq\dim_P(\mathcal{S}_m(w))\leq \max_{1\leq k\leq m} \dim_P(\mathcal{S}_k^{\ast}(w)).
\]
For $k=1$ this is clearly bounded from above by $1$,
and by our observations above for $2\leq k\leq m$ as well when $w$ is close enough to $1$.
Similarly as above via \eqref{eq:thor} we conclude
\[
\dim_H(\mathcal{S}_m(w)+\mathcal{S}_m(w))
\leq 
\dim_P(\mathcal{S}_m(w))+\dim_H(\mathcal{S}_m(w))\leq 2.
\]

\subsection{Proof of Theorem~\ref{t3}}

The claims follow easily from the next lemma. 

\begin{lemma} \label{le}
	Let $\nu_0>1, \nu_1>1$ and $\Lambda=\nu_0\nu_1$.
	If $w_0, w_1$ satisfy $w_i<y_i:=(\nu_i-1)/\Lambda$
	for $i=0,1$, then 
	we have the identity of sets 
	\begin{equation}  \label{eq:G}
	\mathcal{S}_{m,\ut_0}^{(b)}(w_0) + \mathcal{S}_{m,\ut_1}^{(b)}(w_1) =
	\mathcal{S}_{m,\ut_0}(w_0) + \mathcal{S}_{m,\ut_1}(w_1)=\Rm.
	\end{equation}
	Morever, the sets
		\begin{equation}  \label{eq:G5}
	\mathcal{S}_{m,\ut_0}^{(b)\ast}(w_0) + \mathcal{S}_{m,\ut_1}^{(b)\ast}(w_1), \qquad 
	\mathcal{S}_{m,\ut_0}^{\ast}(w_0) + \mathcal{S}_{m,\ut_1}^{\ast}(w_1)
	\end{equation}
	still have full $m$-dimensional Lebesgue measure.
\end{lemma}

Assume the lemma is true. Then putting
\[
\nu_0=\frac{1+\sqrt{w}}{1-w}>1, \qquad \nu_1=\frac{1}{\sqrt{w}}>1,
\]
a calculation verifies $y_0=w$, $y_1=w^{\prime}$ and thus the sets
in \eqref{eq:G0} equal $\Rm$, and those in \eqref{eq:fullm} still
have full measure. 
To finish the proof of the theorem,
we verify Lemma~\ref{le}. 

Let $\ux=(\ux_1,\ldots,\ux_m)\in\Rm$ arbitrary. 
For $i=0,1$, 
we have to find $\underline{x}_i\in \mathcal{S}_{m,\ut_i}^{(b)}(w_i)$
for every $w_i<y_i$,
that sum up to $\ux$, with $y_i=(\nu_i-1)/\Lambda$
where $\Lambda=\nu_0\nu_1$. The full measure claim will
be treated later.
We first assume the homogeneous case 
$\ut_1=\ut_1=\underline{0}$ and later explain 
how to pass to general $\ut_i\in\Rm$.

Define a sequence of intervals $I_j=[g_j,h_j]\cap \mathbb{Z}$, $j\geq 0$, that partitions the positive integers, inductively as follows: Take $M>0$ a large integer. Then set $g_0=1, h_0=M, g_1=M+1$ and for $j\geq 1$ let
\[
h_j= \lfloor \nu_i g_j\rfloor, \qquad g_{j+1}=h_j+1,
\]
where $i\in\{0,1\}$ is so that $j\equiv i\bmod 2$. Notice $h_j/g_j=\nu_i+o(1)$.

Let $b\geq 2$ be any integer.
Now consider the vector $\underline{x}_0=(x_{0,1},\ldots,x_{0,m})\in\Rm$
derived from $\ux$ as follows. For each $x_{0,\ell}$ the base $b$ representation
has digit has $0$ in intervals $I_j$ for even $j$ and the same digits
as $\xi_{\ell}$ for odd $j$, for $1\leq \ell\leq m$. Conversely, let $\underline{x}_1=(x_{1,1},\ldots,x_{1,m})$ where each $x_{1,\ell}$
has base $b$ digit has $0$ in intervals $I_j$ for odd $j$ and the same digits
as $\xi_{\ell}$ for even $j$. Obviously $\underline{x}_0+\underline{x}_1=\ux$.
We must show that $\underline{x}_i\in \mathcal{S}_{m}^{(b)}(w_i)$ for every $w_i<y_i$, to finish the proof of \eqref{eq:G0}
 for the homogeneous case.

Let $Q=b^R$ be any large parameter. For $i=0,1$,
let $j=j(Q,i)$ be the largest index 
with $j\not\equiv i\bmod 2$ so that $R\geq h_{j}$.
Thus $b^{h_j}\leq Q\ll b^{h_{j+2}}\ll b^{\Lambda h_j}$.
By construction, the first non-zero base $b$ digit 
of $b^{h_j}x_{i,\ell}$ after the comma is not before position $h_{j+1}-h_{j}$,
in each coordinate $1\leq \ell\leq m$. 
Hence, if we define for $i=0,1$
the integer vectors $\underline{p}_j=\lfloor b^{h_j}\underline{x}_i\rfloor$,
where we mean that the floor function is applied in each coordinate, then
\[
\Vert b^{h_j}\underline{x}_i - \underline{p}_j\Vert  \ll b^{-(h_{j+1}-h_j) } \ll
b^{-(\nu_i-1) h_j }\ll Q^{-\frac{\nu_i-1}{ \Lambda  }}, \qquad i=0,1.
\]
Since $Q$ was arbitrary indeed $\wo_m^{(b)}(\underline{x}_i)\geq (\nu_i-1)/\Lambda=y_i$, hence $\underline{x}_i\in \mathcal{S}_m^{(b)}(w_i)$ for the induced
$\underline{x}_i$ and any $w_i<y_i$, as required.
The full measure claim \eqref{eq:fullm} 
of the theorem follows from Lemma~\ref{daslemma} applied
to $\mathcal{A}$ the union over intervals $I_j$
with $j\not\equiv i\bmod 2$ and the complement $\mathcal{B}$ the union
over those $I_j$ with $j\equiv i\bmod 2$, and $K=\Rm$.
Indeed, it directly implies that Lebesgue almost all $\ux\in\Rm$ 
in fact induce $\underline{x}_i\in \mathcal{S}_{m,\ut_i}^{(b)\ast}(w_i)$,
$i=0,1$.

Now we explain the modification for general $\ut_i=(\theta_{i,1},\ldots,\theta_{i,m})$, $i=0,1$. Instead of 
letting the base $b$ digits of the $x_{i,\ell}$, $1\leq \ell\leq m$, be $0$ in the respective intervals for $j\equiv i\bmod 2$, 
in these intervals we follow the base $b$ representation of $\theta_{i,\ell}$ from its start. 
We give more details. Without loss of generality assume
$\ut_i\in[0,1)^m, \ux_i\in[0,1)^m$ and let
\[
\theta_{i,\ell}= (0.a_{i,1}^{\ell}a_{i,2}^{\ell}\ldots)_b=
\frac{a_{i,1}^{\ell}}{b}+ \frac{a_{i,2}^{\ell}}{b^2} + \cdots , \qquad\qquad i\in\{ 0,1\},\;\; 1\leq \ell\leq m,
\]
and
\[
\xi_{i,\ell}= (0.d_{i,1}^{\ell}d_{i,2}^{\ell}\ldots)_b= \frac{d_{i,1}^{\ell}}{b}+ \frac{d_{i,2}^{\ell}}{b^2} + \cdots , \qquad\qquad i\in\{ 0,1\},\;\; 1\leq \ell\leq m,
\]
be the base $b$ representation of $\theta_{i,\ell}$ and 
$\xi_{i,\ell}$, respectively.
For $i=0,1$ we define $\underline{\alpha}_i=(\alpha_{i,1},\ldots,\alpha_{i,m})$
with coordinates $\alpha_{i,\ell}\in\mathbb{R}$ for $1\leq \ell\leq m$
defined as follows.
For every $j\equiv i\bmod 2$, 
in $I_j=\{ g_j, g_j +1 , \ldots, h_j\}$ as above, we put
the base $b$ digit of $\alpha_{i,\ell}$ at any position $g_j+u-1\in I_j$ 
equal to the $u$-th base $b$
digit $a_{i,u}^{\ell}$ of $\theta_{i,\ell}$.
Define further for $i=0,1$ and $1\leq \ell\leq m$ the rational numbers
\[
r_{j}^{i,\ell}= \frac{ d_{i,g_j}^{\ell}b^{h_{j-1}} + \cdots+ d_{i,h_j}^{\ell}b^{h_j} }{b^{h_j}}- \frac{ a_{i,1}^{\ell}b^{h_{j-1}} + \cdots+ a_{i,|I_j|}^{\ell}b^{h_j} }{b^{h_j}}\in\mathbb{Q},\quad\; j\geq 1.
\]
Note that $r_{j}^{i,\ell}$ has common denominator $b^{h_j}$. 
To the numbers $\alpha_{i,\ell}$ we add the sum of
$r_j^{1-i,\ell}$ over each $j\equiv i\bmod 2$ to derive $x_{i,\ell}$, i.e.
\[
x_{i,\ell}= \alpha_{i,\ell}+ \sum_{j\geq 1,\\ j\equiv i\bmod 2} r_{j}^{1-i,\ell}=\alpha_{i,\ell}+ \sum_{j\geq 1,\\ j\not\equiv i\bmod 2} r_{j}^{i,\ell}, \qquad \quad i=0,1, \; 1\leq \ell\leq m.
\] 
If we write $\underline{x}_i=(x_{i,1},\ldots,x_{i,m})$, then by construction $\underline{x}_0+\underline{x}_1=\ux$.
Moreover, defining $Q$ and the induced $j$ as above, similarly as above 
the base $b$ digits of $b^{h_j}\underline{x}_i-\underline{p}_{i,j}$ 
for $\underline{p}_{i,j}=\lfloor b^{h_j}\underline{x}_i-\ut_i\rfloor$
in any coordinate vanish up to position $h_{j+1}-h_j$.
Thus we have
\[
\Vert b^{h_j}\underline{x}_i - \underline{p}_{i,j}- \ut_i\Vert \ll b^{-(h_{j+1}-h_j) }  \ll
b^{-(\nu_i-1) h_j }\ll Q^{-\frac{\nu_i-1}{ \Lambda }  },
\]
for $i\in\{0,1\}$, so again $\wo_{m,\ut_i}^{(b)}(\underline{x}_i)\geq (\nu_i-1)/\Lambda=y_i$ and thus, as above $\underline{x}_i\in \mathcal{S}_{m,\ut_i}^{(b)}(w_i)$ for any $w_i<y_i$. 
For the full measure measure claim we again apply Lemma~\ref{daslemma}.
Lemma~\ref{le} and thus the theorem are proved. We mention 
that the special case \eqref{eq:B} of Corollary~\ref{mko}
corresponds to $\nu_0=\nu_1=2$.

\subsection{Proof of Corollary~\ref{coro}}  \label{pko}

   Let $w=1/m+\varepsilon_0$ and
   $w^{\prime}=w+1-2\sqrt{w}=(\sqrt{m}-1)^2/m-\varepsilon_1$ for small $\varepsilon_0>0, \varepsilon_1>0$,
   so that identity \eqref{eq:idne}
   in Theorem~\ref{t3} holds.
	Now the extension~\cite[Theorem~1.5]{dfsu1} of
	Theorem~\ref{DFSU} implies that $w\to \dim_P(\mathcal{S}_m^{\ast}(w))$ is right-continuous
	at $w=1/m$. Hence, since the $\varepsilon_i$ can be chosen
	arbitrarily small, from the sets in \eqref{eq:fullm} having full measure and a rearrangement of \eqref{eq:thor} we get
	\begin{align*}
	\dim_H(\mathcal{S}_{m}^{(b)\ast}(w_1)) &\geq m - \dim_P(\mathcal{S}_m^{(b)\ast})\\ &\geq
	  m - \dim_P(\mathcal{S}_m)=
	m - (m-1+\frac{1}{m+1})= 1-\frac{1}{m+1},
	\end{align*}
	for any $w_1<w^{\prime}$, the claim of the lemma.

\subsection{Proof of Theorem~\ref{H}}

We modify the construction of the proof of Theorem~\ref{t3}.
Let $w^{\prime}>0$ and $\nu_0>1, \nu_1>1$ be real numbers to be 
chosen later related by the identity
\begin{equation}  \label{eq:satisfy}
w^{\prime}=\frac{\nu_1-1}{\nu_0\nu_1}.
\end{equation}
Take any $w<w^{\prime}$. We show that the set
\begin{equation} \label{eq:RITZ}
 \mathcal{S}_{m,\ut}^{(b)\ast}(w) + \mathcal{W}_m^{(b)}(\nu_0-1)
\end{equation}
has full $m$-dimensional Lebesgue measure (in fact removing the $\ast$
the sumset equals $\Rm$).
Provided this is true, from \eqref{eq:thor} and formula \eqref{eq:jarnik} we get
\begin{align*}
\dim_{P}( \mathcal{S}_{m,\ut}^{(b)\ast}(w) ) 
\geq \dim_H(\mathcal{S}_{m,\ut}^{(b)\ast}(w) + \mathcal{W}_m^{(b)}(\nu_0-1)) - \dim_H(\mathcal{W}_m^{(b)}(\nu_0-1)) \geq m- \frac{m}{\nu_0}.
\end{align*}
If we choose $\nu_1$ large enough then $\nu_0=(\nu_1-1)/(w^{\prime}\nu_1)$ will be arbitrarily close
to $w^{\prime -1}$, and since $w$ can be taken arbitrarily
close to $w^{\prime}$, the lower bound $m-mw$ follows for $w>0$.
Finally, if $w=0$, it follows from an obvious inclusion argument.

To show \eqref{eq:RITZ},
we again first restrict to the homogeneous case $\ut=\underline{0}$, 
and describe the generalization later.
Take an arbitrary 
vector $\ux=(\xi_1,\ldots,\xi_m)\in \mathbb{R}^{m}$. 
Consider the intervals $I_j$ partitioning the positive integers 
constructed as follows:
Set $g_0=1, h_0=M, g_1=M+1$ and for $j\geq 1$ let
\[
h_j= \lfloor \nu_i g_j\rfloor, \qquad g_{j+1}=h_j+1,
\]
where $i\in\{0,1\}$ is taken so that $j\equiv i\bmod 2$, i.e.
the same parity as $j$. Notice $h_j/g_j=\nu_i+o(1)$.
Now define $\underline{x}_0=(x_{0,1},\ldots,x_{0,m})$ as follows:
For $1\leq \ell\leq m$, let the digits of $x_{0,\ell}$ be $0$ 
in intervals $I_j$ with $j\equiv 0\bmod 2$, and equal to those of $\xi_{\ell}$
if $j\equiv 1 \bmod 2$.
Conversely, let the coordinates $x_{1,\ell}$ of $\underline{x}_1=(x_{1,1},\ldots,x_{1,m})$
have digit $0$ in intervals $I_j$ when $j\equiv 1\bmod 2$, and equal to those of $\xi_{\ell}$
if $j\equiv 0 \bmod 2$, for $1\leq \ell\leq m$.
Then clearly $\ux=\underline{x}_0+\underline{x}_1$.
To finish the proof of the homogeneous case, we need to show that 
\begin{equation}  \label{eq:OP}
\wo_m^{(b)}(\underline{x}_1) \geq w^{\prime}, \qquad \omega_m^{(b)}(\underline{x}_0)\geq \nu_0-1,
\end{equation}
as by definition $\underline{x}_0\in \mathcal{W}_m^{(b)}(\nu_0-1)$ and 
$\underline{x}_1\in \mathcal{S}_m^{(b)}(w)$ for any $w<w^{\prime}$.
By Lemma~\ref{daslemma} we still have $\underline{x}_1\in \mathcal{S}_m^{(b)\ast}(w)$ for any $w<w^{\prime}$ induced by almost all $\ux\in\Rm$ we started with.

We proceed similar to the proof of Theorem~\ref{t3}. Let $Q=b^{R}$ be large and $j$ be the largest 
even integer such that $b^{h_j}\leq Q$. Then $Q\ll b^{h_{j+2}}\ll b^{\Lambda h_j}$, where $\Lambda=\nu_0\nu_1$.
It is again not hard to see that again not before position $h_{j+1}-h_j$
non-zero digits of $b^{h_j}\underline{x}_1$ after the comma occur, hence writing $\underline{p}_{1,j}= \lfloor b^{h_j}\underline{x}_1\rfloor$, we infer
\[
\Vert b^{h_j}\underline{x}_1 - \underline{p}_{1,j}\Vert \ll b^{-(h_{j+1}-h_j) } \ll
b^{-(\nu_1-1) h_j }\ll Q^{-\frac{\nu_1-1}{ \Lambda  }}=Q^{-w^{\prime}}.
\]
This shows the left inequality in \eqref{eq:OP}. Similarly for odd $j$ with $\underline{p}_{0,j}= \lfloor b^{h_j}\underline{x}_0\rfloor$ we have
\[
\Vert b^{h_j}\underline{x}_0 - \underline{p}_{0,j}\Vert \ll b^{-(h_{j+1}-h_j) } \ll
b^{-(\nu_0-1) h_j }= (b^{h_j})^{-(\nu_0-1)}.
\]
This shows the right inequality in \eqref{eq:OP}.

Finally, if $\ut$ does not vanish, in the intervals $I_j$
with odd $j$, we follow with $\underline{x}_1$
the initial $b$-ary expansion of $\ut$ instead of letting 
the digits be $0$,
and then alter $\underline{x}_0$ accordingly in the same intervals 
to preserve $\underline{x}_0+\underline{x}_1=\ux$, very similarly
to the proof of Theorem~\ref{t3}. We conclude likewise that
$\underline{x}_0\in \mathcal{W}_m^{(b)}(\nu_0-1)$ and almost all $\ux\in\Rm$ induce $\underline{x}_1\in \mathcal{S}_{m,\ut}^{(b)\ast}(w)$,
thus the set in \eqref{eq:RITZ} has full measure in the general case.

\begin{remark}  \label{hirsch}
	Define $b$-ary sets 
	$\widetilde{\mathcal{S}}_{m,\ut}^{(b)\ast}(w)\subseteq \mathcal{S}_{m,\ut}^{(b)\ast}(w)$
	of exact singular order $w$ by
	\[
	\widetilde{\mathcal{S}}_{m,\ut}^{(b)\ast}(w)= \mathcal{S}_{m,\ut}^{(b)\ast}(w) \setminus \bigcup_{t>w} \mathcal{S}_{m,\ut}^{(b)\ast}(t)= \{ \ux\in \mathcal{S}_{m,\ut}^{(b)\ast}(w): \wo_{m,\ut}^{(b)\ast}(\ux)=w \}.
	\]
	It is not hard to show that the sumset 
	$\widetilde{\mathcal{S}}_{m,\ut}^{(b)\ast}(w) + \mathcal{W}_m^{(b)}(\nu_0-1)$ with $\nu_0$
	as in the proof above still contains a full measure
	subset of the set 
	$\widetilde{\mathcal{W}}_m^{(b)}(0)=\{ \ux\in\Rm: \om_m^{(b)}(\ux)=0 \}$ (and dropping the ''$\ast$'' the entire set, proceed
	similar to the proof of \cite[Theorem~3.1]{arxiv}). 
	The latter set has full $m$-dimensional
	Lebesgue measure as follows from more general results of Fraenkel~\cite{franky} (see
	also \eqref{eq:jarnik} obtained in~\cite{bf72}). Thus the claim of Theorem~\ref{H} for the smaller sets $\widetilde{\mathcal{S}}_{m,\ut}^{(b)\ast}(w)$
	follows from the same line of arguments in the proof above. 
	However, it seems
	considerably more difficult to get analogous results for the accordingly altered non-$b$-ary sets $\widetilde{\mathcal{S}}_{m,\ut}^{\ast}(w)$. Similarly, 
	other claims may be extended to exact order of approximation.
\end{remark}

\subsection{Proof of Theorem~\ref{simu}}

Let $\ux\in\Rm$ and $\Theta=\{ \ut_1,\ldots,\ut_k \}$ with $\ut_s\in\Rm$
for $1\leq s\leq k$ 
arbitrary and $\nu_0>1, \nu_1>1$ real numbers. 
For $i=0,1$ and $j\geq 1$ 
define again iteratively a partition of the positive integers
into intervals $I_j=[g_j,h_j]\cap \mathbb{Z}$ with $h_j= \lfloor\nu_i g_j\rfloor$
and $g_{j+1}=h_j+1$, 
where $i=1$ if $j\not\equiv 0\bmod k$ and $i=0$ 
if $j\equiv 0\bmod (k+1)$. 
Then assume in $I_j$ the real vector  
$\underline{x}_1$ follows the base $b$ representation
of $\ut_s$ where $j\equiv s\bmod k$ with $s\in\{ 1,2,\ldots,k\}$ 
for $j\not\equiv 0\bmod (k+1)$,
and the digits are those of $\ux$ in intervals 
with $j\equiv 0\bmod (k+1)$.   
Define $\underline{x}_0=\ux-\underline{x}_1$ and note
that it has $0$ base $b$ digits in $I_j$ for $j\equiv 0\bmod (k+1)$. Then for any $1\leq s\leq k$ and any large $Q_s$, 
let $j=j(s,Q_s)$ be the largest index that satisfies 
$j+1\equiv s\bmod (k+1)$ 
and $b^{h_{j}}\leq Q$. Then with $\Lambda=\nu_0\nu_1^k$
and $\underline{p}_{1,j}=\lfloor b^{h_j}\underline{x}_1\rfloor$,
we have $b^{h_j}\leq Q<b^{h_{j+k+1}}\ll b^{\Lambda h_j }$
and we check similar to the proof above that
\[
\Vert b^{h_j}\underline{x}_1-\underline{p}_{1,j}-\ut_s\Vert \ll
b^{-(h_{j+1}-h_j)} \ll b^{-(\nu_1-1)h_j}
\ll Q^{-\frac{\nu_1-1}{\Lambda}}. 
\]
Moreover, with 
$\underline{p}_{0,j}=\lfloor b^{h_j}\underline{x}_0\rfloor$,
for any $j\equiv -1\bmod (k+1)$ we have
\[
\Vert b^{h_j}\underline{x}_0-\underline{p}_{0,j}\Vert \ll
b^{-(h_{j+1}-h_j)} \ll (b^{h_j})^{-(\nu_0-1)}. 
\]
By construction
$\wo_m^{(b)}(\underline{x}_1,\ut_s)\geq (\nu_1-1)/\Lambda= w^{\prime}$
for any $s\in\{ 1,2,\ldots,k\}$. Thus 
we have $\underline{x}_1\in \cap_{\ut\in\Theta} \mathcal{S}_{m,\ut}^{(b)}(w)$ for any $w<w^{\prime}$, and by
Lemma~\ref{daslemma} still $\underline{x}_1\in \cap_{\ut\in\Theta} \mathcal{S}_{m,\ut}^{(b)\ast}(w)$ for almost all $\ux$ we started with.
Finally choosing $\nu_1=\frac{k}{k-1}$ and if we identify
\[
w^{\prime}=\frac{\nu_1-1}{\Lambda}
\]
we also have
\[
\om_m^{(b)}(\underline{x}_0)\geq \nu_0-1=\frac{\nu_1-1}{w^{\prime}\nu_1^k}-1=\frac{(k-1)^{k-1}}{ w^{\prime}k^k}-1>\frac{1}{kew^{\prime}}-1,
\] 
where we used the inequality
\[
(1-\frac{1}{k})^{k-1} > e^{-1}, \qquad k\geq 2 \quad\; \Longleftrightarrow \quad\;
(1+\frac{1}{k-1})^{k-1}<e, \qquad k\geq 2,
\]
a well-known fact. Hence if we write 
\[
A=\mathcal{W}_m^{(b)}((kew^{\prime})^{-1}-1),\qquad B=\bigcap_{\ut\in \Theta} \mathcal{S}_{m,\ut}^{(b)\ast}(w),
\]
then again $A+B$ has full $m$-dimensional Lebesgue measure for any $w<w^{\prime}$, and
by \eqref{eq:thor} and formula \eqref{eq:jarnik} we infer
\begin{align*}
m=\dim_H(A+B)\leq \dim_H(A)+\dim_P(B)\leq mkew^{\prime}+ \dim_P(B). 
\end{align*}
By a continuity argument again we may write $w$ instead of $w^{\prime}$, and subtracting $mkew$ from both sides the proof is finished.

\subsection{Proof of Theorem~\ref{heep} }

Let $h_j/g_j\to\infty$ and $g_{j+1}=h_j+1$ and define again the sequence
of neighboring intervals $I_j=[g_j,h_j]\cap \mathbb{Z}$
that partitions $\mathbb{N}$. Further take any map 
$\varphi: \mathbb{N}\to \mathbb{N}$ such that each $s\in\mathbb{N}$
has infinitely many preimages both among 
even and odd $j$, for example by $\{ \varphi(1), \varphi(2),\ldots \}=\{1,1,2,1,2,3,1,2,3,4,1,\ldots\}$. Then, for any $s\geq 1$, 
assume $\underline{x}_0$ follows the base $b$ 
representation of $\ut_s$ in an interval $I_j$ whenever $\varphi(j)=s$
and $j$ is even. Similarly, let $\underline{x}_1$ follow
when $j$ is odd and $\varphi(j)=s$.
For given $\ux\in\Rm$, define the digits in the remaining intervals 
so that $\underline{x}_0+\underline{x}_1=\ux$, very similar
to the proof of Theorem~\ref{t3}. Then
for arbitrarily large $N$, any $s\geq 1$ and $j\in \{ \varphi^{-1}(s)-1\}$ with $j$ large enough, for $p_{i,j}=\lfloor b^{h_j}\underline{x}_i \rfloor$ 
we have
\[
\Vert b^{h_j}\underline{x}_i-\underline{p}_{i,j}-\ut_s\Vert
\ll b^{-(h_{j+1}-h_j) }\ll (b^{h_j})^{-N}, \qquad i=0,1.
\]
Hence $\om_{m,\ut_s}^{(b)}(\underline{x}_i)=\infty$ for $i=0,1$ and all
$s\geq 1$. Thus
with $A=\cap_{\ut\in\Theta} \mathcal{W}_{m,\ut}^{(b)}(\infty)$ we have $A+A=\Rm$. We may without loss of generality assume $\underline{0}\in\Theta$
so that $\dim_H(A)=0$ by \eqref{eq:jarnik}.
Consequently from \eqref{eq:thor} we obtain
\[
m=\dim_H(\Rm)=\dim_H(A+A)\leq \dim_H(A)+ \dim_P(A)= \dim_P(A),
\]
hence $\dim_P(A)=m$.

\subsection{Proof of Theorem~\ref{reverse}}

The claim follows from combination of two results below.

\begin{proposition}  \label{prop}
	For any $\xi$ and $\theta_i$, $i=1,2$, as in the theorem we have
	\[
	\wo_{1,\theta_i}^{(3)}(\xi)\leq 1, \qquad \om_{1,\theta_i}^{(3)}(\xi) \geq \frac{\wo_{1,\theta_i}^{(3)}(\xi)}{1-\wo_{1,\theta_i}^{(3)}(\xi)},
	\]
	where we interpret $1/0=\infty$.
\end{proposition}

\begin{proof}
	We only show the claim for $i=1$, the other case is analogous. 
	We show the equivalent claim
	\[
	\wo_{1,\theta_1}^{(3)}(\xi) \leq \frac{\om_{1,\theta_1}^{(3)}(\xi)}{\om_{1,\theta_1}^{(3)}(\xi)+1}\leq 1.
	\]
	We may assume $\om_{1,\theta_1}^{(3)}(\xi)>0$, otherwise
	the claim is clear by the trivial estimate $\wo_{1,\theta_1}^{(3)}(\xi)\leq \om_{1,\theta_1}^{(3)}(\xi)$.
	By definition of $\om_{1,\theta_1}^{(3)}(\xi)$, for any $\tau\in (0,\om_{1,\theta_1}^{(3)}(\xi))$, 
	arbitrarily large integers $t$ and $p_1=p_1(t)$ we have
		\begin{equation}  \label{eq:tauin}
	|3^t\xi-p_1-\theta_1| < (3^t)^{-\tau}.
	\end{equation}
	By an easy digital argument,
	using that only at the beginning of the expansion of $\theta_1$ we have two consecutive base $3$ digits $1$,
	for $u\in (t,\lfloor t (\tau+1)\rfloor]$ we
	have $|3^{u}\xi-p_5-\theta_1|\gg 1$. 
	For $u\leq t$, arbitrarily small $\varepsilon>0$ and any integer $p_2$ we estimate
	\[
	|3^{u}\xi-p_2-\theta_1|\gg (3^{u})^{-(\om_{1,\theta_1}^{(3)}(\xi) +\varepsilon) }\geq 
	(3^{t})^{-(\om_{1,\theta_1}^{(3)}(\xi) +\varepsilon) }
	\]
	from the definition of $\om_{1,\theta_1}^{(3)}(\xi)$.
	Hence, the last estimate holds for any $u\leq \lfloor t (\tau+1)\rfloor$. Hence
	with $Q=3^{\lfloor t(\tau+1)\rfloor}$ we get
	\[
	\wo_{1,\theta_1}^{(3)}(\xi)\leq \frac{t(\om_{1,\theta_1}(\xi)+\varepsilon)}{t(\tau+1)}=
	\frac{\om_{1,\theta_1}(\xi)+\varepsilon}{\tau+1}.
	\]
	Now $\tau$ can be chosen arbitrarily 
	close to $\om_{1,\theta_1}^{(3)}(\xi)$ and $\varepsilon$ arbitrarily
	small and the claim follows.
\end{proof}

\begin{proposition}  \label{pp1}
	For $\theta_1, \theta_2$ as in the theorem and any  $\xi\in\mathbb{R}$ we have 
	\[
	\wo_{1,\theta_2}^{(3)}(\xi)\leq \frac{1}{\om_{1,\theta_1}^{(3)}(\xi)+1}\leq 1.
	\]
\end{proposition}

\begin{proof}
	We may assume $\om_{1,\theta_1}(\xi)>0$ as otherwise the claim
	follows from Proposition~\ref{prop}.
	By definition, for any $\tau\in (0,\om_{1,\theta_1}^{(3)}(\xi))$ 
	and arbitrarily large integers $t, p_1=p_1(t)$ we have
    \eqref{eq:tauin}.
	Then the base $3$ digits of $\xi$ at 
	places $t+1,t+2,\ldots, \lfloor t(\tau+1)\rfloor$ are the
	same as the initial digit sequence of $\theta_1$, which
	has digit $1$ at places of the form $N!$ and $0$ otherwise.
	Thus, since $\theta_2$ starts with very different
	digits $(0.220)_{3}$ in base $3$,
	it is easy to see that for any integer $u\in [t,\lfloor t (\tau+1)\rfloor]$ and any integer $p_2$ we have
	\[
	|3^u\xi-p_2-\theta_2|\gg 1.
	\]
	For similar reasons, if $u<t$ then for any integer $p_3$ we have 
	\[
	|3^u\xi-p_3-\theta_2|\gg 3^{u-t}.
	\]
	Hence for $u\leq \lfloor t(\tau+1)\rfloor$ we have $\vert 3^u\xi-p_4-\theta_2|\gg 3^{u-t}\gg 3^{-t}$,
	thus considering again $Q=3^{\lfloor t(\tau+1)\rfloor}$ we get $\wo_{1,\theta_2}^{(3)}(\xi)\leq \frac{t}{t(\tau+1)}=\frac{1}{\tau+1}$. 
	The claim follows as $\tau$ can be chosen arbitrarily 
	close to $\om_{1,\theta_1}^{(3)}(\xi)$.
	\end{proof}

We deduce the first claim of the theorem by combining the propositions via
\[
\wo_{1,\theta_2}^{(3)}(\xi)\leq \frac{1}{\om_{1,\theta_1}^{(3)}(\xi)+1}\leq \frac{1}{\frac{\wo_{1,\theta_1}^{(3)}(\xi)}{1-\wo_{1,\theta_1}^{(3)}(\xi)}+1}= 1-\wo_{1,\theta_1}^{(3)}(\xi).
\]
The latter claim of the theorem
follows since by the first not both $\wo_{1,\theta_i}^{(3)}(\xi)$
can exceed $1/2$.

\subsection{Proof of Theorem~\ref{cantor}}

Additionally to Lemma~\ref{daslemma},
we further require upper bounds on the Hausdorff dimension
of the set of $b$-ary, ordinary $\tau$-approximable vectors in $K$.
A supposedly sharp bound reads as follows.

\begin{lemma}  \label{lemus}
	Let $m\geq 1$ be an integer and $K$ as in \eqref{eq:wiro}. We have
	\[
	\dim_H(\mathcal{W}_m^{(b)}(\tau)\cap K) \leq \frac{\dim(K)}{\tau+1}= \frac{\sum_{i=1}^{m} \log |W_i|}{(\tau+1)\log b}, \qquad \tau\in[0,\infty].
	\]
\end{lemma}

The proof of Lemma~\ref{lemus} follows from a very standard Borel-Cantelli
covering approach, mimicking the short proof of the 
special case $m=1$, $K=C_{3,\{0,2\}}$
in~\cite[\S~5]{lsv}. We leave the details to the reader.

\begin{remark}
		Lemma~\ref{lemus} can alternatively proved with Proposition~\ref{falke} below
	from~\cite{falconer}.
	Showing equality in Lemma~\ref{lemus} is a considerably harder task. It was proved rigorously for $m=1$
	and $K=C_{3,\{0,2\}}$ the Cantor middle third set by Levesley, Salp, Velani~\cite{lsv}. The right identity in \eqref{eq:jarnik} is 
	a special case 
	when we let $W=\{ 0,1,\ldots,b-1\}$. 
\end{remark}

Now we prove Theorem~\ref{cantor}.
Assume first $0\in W_h$ for all $1\leq h\leq m$. 
Starting with $\ux\in K$, then for $\ut=\underline{0}$ 
the contructions in Theorem~\ref{H} and Theorem~\ref{t3} obviously
lead to $\underline{x}_i\in K$, $i=0,1$, as well. Thus, following the proof of Theorem~\ref{H} and letting
\[
A(w)= \mathcal{S}_{m}^{(b)}(w) \cap K, \qquad B= \mathcal{W}_m^{(b)}(\nu_0-1) \cap K,
\]
with the same assumptions and notation,
we get the inclusion
\begin{equation}  \label{eq:EE}
A(w)+B\supseteq K.
\end{equation}
We next show that for 
$A^{\ast}(w)=A(w)\cap \Rr$ the restriction of $A(w)$ to $\Rr$, the set
$(A^{\ast}(w)+B)\cap K$
still has full Cantor measure, in particular
 \begin{equation} \label{eq:FF}
 \dim_H((A^{\ast}(w)+B)\cap K)= \dim(K).
 \end{equation}
This is immediate from \eqref{eq:EE}
and Lemma~\ref{daslemma}
with $\mathcal{A}$ the union of the $I_j$ over $j\not\equiv i\bmod 2$,
as it precisely claims that for almost all $\ux\in K$ we started with 
the constructed $\underline{x}_i$ do not lie in a rational
affine hyperplane.
Combining \eqref{eq:FF} with \eqref{eq:thor} and \eqref{eq:kpro}
we get
\begin{align*}
\dim_P(\mathcal{S}_m^{(b)\ast}(w)) &\geq \dim_H( A^{\ast}(w)+ B)- \dim_H(B)\\
&= \dim(K)-\dim_H(B)= \frac{\sum_{h=1}^{m} \log |W_h|}{\log b}-\dim_H(B).
\end{align*}
An upper bound for $\dim_H(B)$ comes from Lemma~\ref{lemus}. 
Our claim \eqref{eq:csr} and thus also \eqref{eq:habicht}
in (i) follow directly with
identification $w^{-1}=\nu_0-1$ and letting $\tau=\nu_0-1$.
As similar argument based on 
the proof of Theorem~\ref{t3} again yields
the inclusion in (ii), and with Lemma~\ref{daslemma} we get the
metrical claim in (ii). 
Finally, if $0\notin W_h$ for some $h$, we reduce \eqref{eq:nonedda} 
to the case $0\in W_h$ in \eqref{eq:csr}
via applying the rational shift $t\to t- \min W_h/(b-1)$ in coordinate $h$. It maps general $C_{b,W}$ into some 
$C_{b,W^{\prime}}$ with $0\in W^{\prime}$ and $|W^{\prime}|=|W|$.
Moreover, it both keeps the
property of belonging to some set $\mathcal{S}_{m}^{\ast}(w)$ 
invariant and, as a bi-Lipschitz map, preserves the Hausdorff
and packing dimensions. The proof is finished.

%

\subsection{Proof of Theorem~\ref{ax3} }

We use~\cite[Example~4.6]{falconer}.

\begin{proposition}[Falconer]  \label{falke}
	Let $[0,1] = E_0 \supseteq E_1 \supseteq E_2 \supseteq \cdots$
	be a decreasing sequence of sets, with each $E_k$ a union of a finite number of
	disjoint closed intervals (called $k$-th level basic intervals), with each interval of
	$E_{k-1}$ containing $m_{k}\geq 2$ intervals of $E_{k}$, 
	which are separated by gaps of length at least $\epsilon_k$,
	with $0<\epsilon_{k+1}<\epsilon_{k}$ for each $k$,
	which tend to $0$ as $k\to\infty$. Then the set
	\[
	F= \bigcap_{i\geq 1} E_i
	\]
	satisfies 
	\[
	\dim_H(F)\geq \liminf_{k\to\infty} \frac{ \log(m_1m_2 \ldots m_{k-1})}{-\log m_k \epsilon_k}.
	\]
\end{proposition}

Let $A\asymp B$ denote $A\ll B\ll A$.
Let $H_n=b^{h_{2n}}$ with $h_j$ as in the notation of the proof
of Theorem~\ref{H}. As in the proof of the homogeneous case
of Theorem~\ref{H},
we consider the set $\mathscr{Y}$, corresponding to coordinates of $\underline{x}_1$,
of real numbers where we can choose the base $b$ digits freely
in intervals $I_{2n}=[h_{2n}, \nu_0 h_{2n}]$, and take digit $0$ in the remaining
intervals of the form $I_{2n+1}=[\nu_0 h_{2n}, h_{2n+2}]$, where
$h_{2n+1}\asymp \nu_0 h_{2n}$ and $h_{2n+2}\asymp \nu_0\nu_1 h_{2n}$.
It fits the Cantor type interval construction of
Proposition~\ref{falke} with parameters
\[
m_k\asymp H_k^{\nu_0-1}, \qquad \epsilon_k\asymp H_k^{-\nu_0}.
\]
It yields after a short calculation via the geometric sum formula that
\[
\dim_H(\mathscr{Y})\geq \liminf_{k\to\infty} 
\frac{ (\nu_0-1)\sum_{j=1}^{k-1} \frac{1}{(\nu_0 \nu_1)^j} h_k }{ h_k } = 
\frac{\nu_0-1}{\nu_0\nu_1-1}.
\]    
One may compare the method with the proof of~\cite[Theorem~2.3]{arxiv3} for more details.
On the other hand,
if $\nu_0, \nu_1$ are related by identity \eqref{eq:satisfy},
then the proof of Theorem~\ref{H} shows
$\mathcal{S}_m^{(b)}(w)\supseteq \mathscr{Y}^m=\mathscr{Y}\times \cdots\times \mathscr{Y}$, and with repeated application of \eqref{eq:kproS} and the aid of Lemma~\ref{daslemma} 
still $\dim_H(\mathcal{S}_m^{(b)\ast}(w))\geq m\dim_H(\mathscr{Y})$. 
Hence $\dim_H(\mathcal{S}_m^{(b)\ast}(w))\geq m\dim_H(\mathscr{Y})\geq m(\nu_0-1)/(\nu_0\nu_1-1)$, where we used that $w^{\prime}$ is arbitrarily close to $w$. 
Upon \eqref{eq:satisfy}, which is basically $\nu_1= (1-\nu_0w)^{-1}$, we maximize this with choice of parameters
\[
\nu_0= \frac{2}{w+1},\qquad \nu_1=\frac{1+w}{1-w}, 
\]
inserting we derive the claimed bound. The extension to the 
general inhomogeneous case
is obtained by very similar ideas as in the proof of Theorem~\ref{H} again.

\subsection{Proof of Theorem~\ref{khr}}

With notation from the proof of Theorem~\ref{ax3}
and Theorem~\ref{khr}, the according sets $\mathscr{Y}_i\subseteq \mathbb{R}$, $1\leq i\leq m$,
corresponding to the coordinates for $\underline{x}_1$ constructed in the proof of Theorem~\ref{cantor} are again as in~Proposition~\ref{falke} with parameters
\[
m_k\asymp H_k^{(\nu_0-1)d_i}, \qquad \epsilon_k\asymp H_k^{-\nu_0}.
\]
Proceeding as in the proof of Theorem~\ref{ax3}
leads to a lower bound of the form
\[
\dim_H(\mathcal{S}_K^{(b)\ast}(w))\geq \sum_{i=1}^{m} \dim_H(\mathscr{Y}_i)\geq  \sum_{i=1}^{m} d_i \cdot 
\frac{ -w\nu_0^2 + (w+1)\nu_0 -1 }{ (w+1)(1-d_i) \nu_0^2 + ((w+2)d_i -1) \nu_0 - d_i }.
\]
The first claim follows when identifying $\nu_0$ with $t$.
If $K=C_{3,\{0,2\}}\times C_{3,\{0,2\}}$, inserting $m=2, d_1=d_2=\log 2/\log 3, w=1/2$ and choosing the optimal parameter $\nu_0=1.2994\ldots$,
we calculate a bound larger than $0.1255$. Finally 
as in the proof of Theorem~\ref{cantor} we can always reduce the problem to the case $0\in W_i$ by a rational transformation to infer
the same estimates for the non-$b$-ary sets $\mathcal{S}_K^{\ast}(w)$.

\subsection{Preparation for topological results}

The proof of Theorem~\ref{T2} relies on  
an observation that is essentially due to Erd\H{o}s~\cite{erdos}.

\begin{lemma}[Erd\H{o}s]  \label{erd}
	If $A,B\subseteq \Rm$ are comeagre, then $A+B=A\cdot B=\Rm$.
\end{lemma}

A short proof for $A=B$ and $m=1$ based on Baire's category Theorem
is described in~\cite{erdos}. We present the analogous proof for
the sum in the general case.
For $\underline{t}\in\Rm$, we may write the translate
$\underline{t}-A= \cap_{j\geq 1} O_{1,j}$ of $-A$
as countable intersection of open dense sets $O_{1,j}=\underline{t}-O_{1,j}^{\prime}$. 
Similarly $B=\cap_{j\geq 1} O_{2,j}$ for open dense sets $O_{2,j}$. Then
the intersection $(\underline{t}-A)\cap B=\cap_{j\geq 1} O_{1,j} \cap_{j\geq 1} O_{2,j}$ is still a countable
intersection of open dense sets, thus by Baire's theorem
dense, in particular non-empty. This means $\underline{t}\in A+B$.
Since $\underline{t}\in\Rm$ was arbitrary, the claim follows.
The proof for the product works very similarly, considering
the componentwise quotient
$\underline{t}/A=\{ \underline{t}\}/A$ in place of $\underline{t}-A$.

\subsection{Proof of Theorem~\ref{T2}}

Since the set $\mathcal{W}_m(\infty)$ is comeagre 
as recalled in \S~\ref{top},
by Lemma~\ref{erd}, if $A\subseteq \Rm$ is a comeagre set as well, then we the sumset
$\mathcal{W}_m(\infty)+ A$ must be the entire space $\Rm$. On the other hand, from \eqref{eq:thor} and \eqref{eq:jarnik}
we see that
\begin{align*}
\dim_H(\mathcal{W}_m(\infty)+ A)\leq 
 \dim_H( \mathcal{W}_m(\infty) ) + \dim_P(A)= \dim_P(A),
\end{align*}
hence indeed we conclude $\dim_P(A)=m$.

\subsection{Proof of Theorem~\ref{C3} }  \label{s5.12}

Let $m\geq 1$. For
$A,B\subseteq \Rm$, write $A\pm B$ for either $A+B$ or 
$A-B$ and $A^2=\{ (a_1^2, \ldots,a_m^2): (a_1, \ldots,a_m)\in A\}$.
Now let $A=\mathcal{W}_m(\infty)$ throughout. 
We have $A\circ A=\Rm$ for any $\circ\in\{ +,-,\cdot,/ \}$ 
by Lemma~\ref{erd} and since $A=-A$ and $A=A^{-1}$. 
The latter identity needs some explanation. Assume $\ux=(\xi_1,\ldots,\xi_m)\in A$.
Then 
\[
\Vert q\ux-\underline{p}\Vert \leq q^{-N}
\]
has a solution for arbitrarily large $N$. Since $|p_i|\asymp q$, $1\leq i\leq m$,
are all of the same magnitude, we check that if we let
$\underline{p}=(p_1,\ldots,p_m)$ and $P=p_1p_2\cdots p_m$ then
\[
\Vert P\xi_i^{-1}-\frac{qP}{p_i}\Vert \ll q^{-N+m-1}\ll
|P|^{-\frac{N-m+1}{m} }, \qquad 1\leq i\leq m.
\]
Since the term $(N-m+1)/m$ tends to infinity with $N$,
we infer $\ux^{-1}=(\xi_1^{-1}, \ldots,\xi_m^{-1})\in A$ as well, and the claim follows.
Moreover, from Theorem~\ref{DFSU}, the argument in the proof of Theorem~\ref{T2}
and since $\mathcal{S}_m=-\mathcal{S}_m$, we see 
\[
\dim_H(A\pm \mathcal{S}_m)\leq \dim_P(\mathcal{S}_m)<m, \qquad m\geq 1.
\]
From \eqref{eq:B},
for $m\geq 5$ and $\epsilon\in(0,1/20)$ we infer $\mathcal{S}_m\pm \mathcal{S}_m\supseteq \mathcal{S}_m(1/4-\epsilon)\pm \mathcal{S}_m(1/4-\epsilon)=\Rm$. 

Now let again $m\geq 1$ be arbitrary.
Write $C=C_{3,\{0,2\}}$  for Cantor's middle
third set and let 
 $B=((C\cup C^2)+\mathbb{Z})^m$, 
the $m$-fold Cartesian product of the union of all
shifts by integers of Cantor's middle third set and its square. 
It is well-known that $C+C=[0,2]$ and $C-C=[-1,1]$,
see~\cite[\S1]{art} for a short proof, and hence $B+B=B-B=\Rm$.
In~\cite[Theorem~1]{art} it is shown that $C^2\cdot C=\{ x^2y: x,y\in C\}=[0,1]$, thus $B\cdot B=\Rm$. Moreover, again as a consequence
of~\cite[Theorem~1]{art}, the set $C/C$
contains arbitrarily large intervals (in fact its precise structure 
as a countable union of closed intervals is determined). 
Hence we have $B/B=\Rm$.

Since $C^2$ is a locally Lipschitz image of $C$ we
have $\dim_H(C^2)=\dim_{H}(C)$ as well as $\dim_P(C^2)=\dim_P(C)$,
hence by \eqref{eq:haupack} these values all coincide, 
and since $B$ is only a countable union, 
from \eqref{eq:kproS} we finally get
\begin{equation}  \label{eq:pp}
\dim_P(B)=\dim_H(B)=\dim(C^m)=\frac{\log 2}{\log 3}m<m.
\end{equation}
Thus, on the other hand, since our componentwise operations $\circ\in\{+,-,\cdot,/\}$ are 
locally Lipschitz maps, by \eqref{eq:tricot}, \eqref{eq:jarnik} and \eqref{eq:pp}
we have 
\[
\dim_H(A\circ B)\leq \dim_H(A\times B) \leq \dim_H(A) + \dim_P(B) \leq 0+\frac{\log 2}{\log 3}m<m.
\]
For $\circ=+$ see alternatively
\eqref{eq:thor}. All claims are proved.

%
%
%
%

\end{document}